\documentclass{article}
\usepackage{fourier}
\usepackage{arcs}
\usepackage{geometry}
\usepackage{amssymb}
\usepackage{amsfonts}
\usepackage{graphicx}
\usepackage{amsmath}
\usepackage{bbm}
\usepackage{stackengine}
\usepackage{epsf,epsfig,subfigure}
\usepackage{float,color,ulem}
\usepackage[colorlinks=true]{hyperref}
\usepackage{graphicx,multirow}
\usepackage{epsf,epsfig,subfigure,diagbox}
\usepackage[toc,page,title,titletoc,header]{appendix}
\usepackage{indentfirst}
\usepackage[english]{babel}
\usepackage{colortbl}

\newfloat{figure}{H}{lof}
\newfloat{table}{H}{lot}
\floatname{figure}{\figurename}
\floatname{table}{\tablename}
\newtheorem{theorem}{Theorem}[section]

\newtheorem{Assumption}[theorem]{Assumption}

\newtheorem{corollary}[theorem]{Corollary}

\newtheorem{lemma}[theorem]{Lemma}

\newtheorem{proposition}[theorem]{Proposition}
\newtheorem{remark}[theorem]{Remark}

\newenvironment{proof}[1][Proof]{\noindent\textit{#1.} }{\hfill \rule{0.5em}{0.5em}}

\numberwithin{equation}{section}
\newcommand{\R}{\mathbb{R}}

\begin{document}

\title{Large Speed Traveling Waves for the Rosenzweig-MacArthur Model with Spatial Diffusion }
\author{\textsc{Arnaud Ducrot$^{a,}${\small \thanks{Research was partially supported by CNRS and  National Natural Science Foundation of China (Grant No.11811530272)}} , Zhihua Liu$^{b,}$ {\small \thanks{Research was partially supported by CNRS and National Natural Science Foundation of China (Grant Nos. 11871007 and 11811530272) and the Fundamental Research Funds for the Central Universities.}}  and Pierre Magal$^{c,*}$} \\
$^{a}${\small \textit{Normandie Univ, UNIHAVRE, LMAH, FR-CNRS-3335, ISCN, 76600 Le Havre, France}}  \\
$^{b}${\small \textit{School of Mathematical Sciences, Beijing Normal
University,}}\\
{\small \textit{Beijing 100875, People's Republic of China }} \\
$^{c}${\small \textit{Univ. Bordeaux, IMB, UMR 5251, F-33400 Talence, France.}} \\
{\small \textit{CNRS, IMB, UMR 5251, F-33400 Talence, France.}} }

\maketitle

\begin{center}
\textit{This article is dedicated to Professor Sze-Bi Hsu.}
\end{center}

\begin{abstract}
This paper focuses on traveling wave solutions for the so-called Rosenzweig-MacArthur model with spatial diffusion. The main results of this note are concerned with the existence and uniqueness of traveling wave solution as well as periodic wave train solution in the large wave speed asymptotic.
Depending on the model parameters we more particularly study the existence and this uniqueness of a traveling wave
connecting two equilibria or connecting an equilibrium point and a periodic wave train. We also discuss the existence and uniqueness of such a periodic wave train. Our analysis is based on ordinary differential techniques by coupling the theories of invariant manifolds together with those of global attractors.
\end{abstract}

\vspace{0.1in}\noindent \textbf{Key words}. Predator prey, periodic wave train, traveling wave, Rosenzweig-MacArthur Model.

\vspace{0.1in}\noindent \textbf{2010 Mathematical Subject Classification}. 35J61, 35C07, 34C23.

\section{Introduction}
In this work we study the traveling solutions for the so-called diffusive Rosenzweig-MacArthur predator-prey system that reads as follows
\begin{equation}\label{1.1}
\left\lbrace
\begin{array}{lll}
u_{t} &=&\delta _{1}u_{xx}+Au\left(1-\dfrac{u}{K}\right)-B\dfrac{uv}{1+Eu},\vspace{0.1cm} \\

v_{t} &=&\delta _{2}v_{xx}-Cv+D\dfrac{uv}{1+Eu}.
\end{array}
\right.
\end{equation}
This system is posed for the one-dimensional spatial variable $x\in\mathbb R$ while $t$ denotes time.

In the above system of equations $u=u(t,x)$ denotes the density of the prey population while $v=v(t,x)$ corresponds to those of the predator, at time $t>0$ and spatial location $x\in\mathbb R$. The positive parameters $\delta_1$ and $\delta_2$ represent the diffusion coefficients for the prey and the predator, respectively. The underlying kinetic system describes the dynamics of the populations as well as their interactions and reads as the following ordinary differential equations (ODE for short)
\begin{equation}\label{1.2}
\left\lbrace
\begin{array}{lll}
u'(t) &=&Au\left(1-\dfrac{u}{K}\right)-B\dfrac{uv}{1+Eu},\\
v'(t) &=&-Cv+D\dfrac{uv}{1+Eu},
\end{array}
\right.
\end{equation}
wherein $A$, $B$, $C$, $D$ and $K$ are given positive constants.
More precisely $A$ stands for the growth factor for the prey species, $K$ denotes its carrying capacity, $B$ and $D$ are the interaction rates for the two species while $C$ corresponds to the natural death rate for the predator. Finally the parameter $E$ measures the "satiation" effect of the predator population. We refer the reader to Holling  \cite{Holling} for more details on this model.

The aim of this work is to discuss the existence and qualitative properties of the traveling wave and the periodic wave train solutions for \eqref{1.1} .
To discuss this issue, we first rescale the system by introducing
\begin{eqnarray*}
U &=&Eu,\;V=Bv/C,\;t^{\prime }=Ct,\;x^{\prime }=(C/\delta _{2})^{1/2}x, \\
d &=&\frac{\delta _{1}}{\delta _{2}},\;\alpha =A/(ECK),\;\gamma =EK,\;\beta =D/(EC).
\end{eqnarray*}
With these new variables and normalized parameters, \eqref{1.1} rewrites, omitting the prime for notational simplicity, as the following reaction-diffusion system
\begin{equation}\label{1.3}
\left\lbrace
\begin{array}{ll}
U_{t} &=dU_{xx}+\alpha U(\gamma -U)-\dfrac{UV}{1+U},   \\
V_{t} &=V_{xx}-V+\beta \dfrac{UV}{1+U},  
\end{array}
\right.
\end{equation}
while the underlying kinetic system, namely \eqref{1.2}, becomes
\begin{equation}\label{1.3bis}
\left\lbrace
\begin{array}{ll}
U' =&\alpha U(\gamma -U)-\dfrac{UV}{1+U},   \\
V' =&-V+\beta \dfrac{UV}{1+U}.
\end{array}
\right.
\end{equation}

As mentioned above, the goal of this work is to discuss some properties of the traveling wave and periodic wave train solutions for the reaction-diffusion system \eqref{1.3}. Here recall that a traveling wave solution corresponds to an entire solution of \eqref{1.3} (that is a solution defined for all time $t\in\mathbb R$) of the form
\begin{equation*}
U(t,x)=u(s),\;\;V(t,x)=v(s)\text{ with }s=x+ct,
\end{equation*}
where $c\in\R$ is some constant that stands for the wave speed. When the profile $s\mapsto (u(s),v(s))$ is periodic we speak about periodic wave train with speed $c$.
Plugging this specific form into \eqref{1.3} yields the following ODE system for the wave profiles $(u,v)=(u(s),v(s))$ for $s\in\R$
\begin{equation}\label{1.4}
\left\lbrace
\begin{array}{lll}
cu^{\prime } &=&du^{\prime \prime }+\alpha u(\gamma -u)-\dfrac{uv}{1+u}, \\
cv^{\prime } &=&v^{\prime \prime }-v+\beta \dfrac{uv}{1+u}.  
\end{array}
\right.
\end{equation}

Traveling wave solutions for the above system or more generally for predator-prey systems have been widely investigated in the last decades. One may refer the reader to the works of Dunbar \cite{Dumbar83, Dunbar84, Dumbar86} who proposed ODE methods coupled topological arguments to prove the existence of such special solutions. One may also refer to Gardner \cite{Gardner84} who developed topological arguments based on Conley index to obtain the existence of solutions with suitable behaviour at $s=\pm\infty$. We also refer to Huang, Lu and Ruan \cite{Huang03} for more general results also based on a coupling between ODE methods and topological arguments. We refer to Ruan \cite{Ruan} a result of existence of periodic wave train by using using Hopf bifurcation method.  We refer the reader to the work of Hosono \cite{Hosono} and the references cited therein for results about the Lotka-Volterra predator-prey system as well as to the recent work of Li and Xiao \cite{Li-Xiao} (see also the references therein) for results about the existence of traveling waves for more general functional responses and also for a nice review on this topic. The connexion between wave solutions and the asymptotic behaviour of the Cauchy problem \eqref{1.3} (when equipped with suitable initial data) has been scarcely studied. One may refer the reader to Gardner \cite{Gardner91} who studied the local stability of wave solutions and to Ducrot, Giletti and Matano \cite{DGM19} (and the references therein) for results related to the so-called asymptotic speed of spread.

One important difficulty when studying traveling wave solutions for predator-prey interactions relies on the ability of the underlying kinetic to generate sustained oscillations, typically through Hopf bifurcation. Hence the behaviour of the solutions of the corresponding reaction-diffusion system are expected to exhibit somehow complex spatio-temporal oscillations. Therefore the traveling wave solutions describing for instance the spatial invasion of a predator is also expected to exhibit oscillating patterns, connecting a predator-free equilibrium and some oscillating state, such as a periodic wave train (see \cite{Huang03, Ruan} for results about the existence of such periodic solutions using bifurcation methods). According to our knowledge, this question related to the shape and the behaviour of traveling waves remains largely open. It has been addressed by Dunbar in \cite{Dumbar86} and further developed by Huang \cite{Huang12}. In this aforementioned work, the author developed refined singular perturbation analysis based on the hyperbolicity of the periodic solutions of the kinetic system to construct oscillating traveling wave in the large speed asymptotic. In this work we revisit this issue by developing a dynamical system approach to obtain a complete picture of the traveling wave solutions for system \eqref{1.4}, in the large wave speed asymptotic, $c\gg 1$. Our methodology also allows us to provide uniqueness results, on the one hand for traveling waves and, on the other hand, also for periodic wave trains with large wave speed.

In this paper, we describe in particular sharp conditions on the parameters of the system that ensure the existence of a unique traveling wave solution for \eqref{1.4} connecting the predator free equilibrium to the interior equilibrium or to a unique periodic wave train. To reach such a refined description, we develop a methodology based on dynamical system arguments. Here we will more precisely couple center manifold and more generally invariant manifold reduction together with the global attractor theory and qualitative analysis for ODE.

To perform our analysis, we make use of successive rescaling arguments to restrict our analysis to a system of two ordinary differential equations. Firstly let us set
\begin{equation}
\widehat{u}(s)=u(-cs),\widehat{v}(s)=v(-cs). \notag
\end{equation}%
Then, dropping the hats on $u$ and $v$ for
notational convenience, \eqref{1.4} becomes
\begin{equation}\label{1.5}
\left\{
\begin{array}{ll}
-u^{\prime } & =\dfrac{d}{c^{2}}u^{\prime \prime }+\alpha u(\gamma -u)- \dfrac{uv}{1+u}, \\
-v^{\prime } & =\dfrac{1}{c^{2}}v^{\prime \prime }-v+\beta \dfrac{uv}{1+u}.
\end{array}
\right.
\end{equation}
Next let us set $\varepsilon =\frac{1}{c^{2}}$, $u_{1}=u,u_{2}=u^{\prime }$, $v_{1}=v$ and $v_{2}=v^{\prime }$, so that the above problem \eqref{1.5} rewrites as
\begin{equation}
\left\{
\begin{array}{lll}
u_{1}^{\prime } & = & u_{2}, \\
d\varepsilon u_{2}^{\prime } & = & -u_{2}-\alpha u_{1}(\gamma -u_{1})+\dfrac{u_{1}v_{1}}{1+u_{1}}, \\
v_{1}^{\prime } & = & v_{2}, \\
\varepsilon v_{2}^{\prime } & = & -v_{2}+v_{1}-\beta \dfrac{u_{1}v_{1}}{1+u_{1}}.
\end{array}
\right.  \label{1.6}
\end{equation}
Set
\begin{equation*}
\widehat{u}_{1}(s)=u_{1}(\varepsilon s),\widehat{u}_{2}(s)=u_{2}(\varepsilon
s),\widehat{v}_{1}(s)=v_{1}(\varepsilon s),\widehat{v}_{2}(s)=v_{2}(\varepsilon s),
\end{equation*}
and \eqref{1.6} becomes (dropping the hats for notational convenience)
\begin{equation}
\left\{
\begin{array}{lll}
u_{1}^{\prime } & = & \varepsilon u_{2}, \\
du_{2}^{\prime } & = & -u_{2}-\alpha u_{1}(\gamma -u_{1})+\dfrac{u_{1}v_{1}}{1+u_{1}}, \\
v_{1}^{\prime } & = & \varepsilon v_{2}, \\
v_{2}^{\prime } & = & -v_{2}+v_{1}-\beta \dfrac{u_{1}v_{1}}{1+u_{1}},
\end{array}
\right.  \label{1.7}
\end{equation}
where all the parameters $d,\alpha,\gamma,\beta$ and $\varepsilon$ are strictly
positive.

As mentioned above, in this paper we will investigate traveling waves and periodic wave trains for \eqref{1.4}, that correspond to  heteroclinic connexions and periodic orbits, respectively, for system \eqref{1.6} or equivalently \eqref{1.7}. Here we focus our study on the large speed asymptotic, namely $c \gg 1$, that is $0<\varepsilon=\frac{1}{c^2}\ll 1$. To study this problem we will use center manifold reduction arguments to rewrite \eqref{1.7} on a suitable invariant set as a small perturbation of the kinetic system \eqref{1.3bis}. The description of the heteroclinic and periodic orbits of the perturbed problem are then investigated using global attractor theory.

The organization of the paper is as follows. Section 2 is devoted to the description of the global attractor for the Rosenzweig-MacArthur model \eqref{1.3bis} with a particular attention paid on the heteroclinic orbits and their uniqueness. Section 3 is concerned with the study of some complete orbit of \eqref{1.7}, in the regime $0<\varepsilon\ll 1$. We first reformulate this problem as a small perturbation of \eqref{1.3bis}. We then study its global attractor and derive existence and uniqueness results for the traveling waves and periodic wave trains for \eqref{1.4} whenever $c$ is large enough. In the last section we present some numerical simulations for the system in order illustrate our results.

\section{Global attractors for the Rosenzweig-MacArthur model}

In this section we propose a refine description of the global attractor of the Rosenzweig-MacArthur
model
\begin{equation}
\left\{
\begin{array}{ll}
U^{\prime }(t)=\alpha U(\gamma -U)-\dfrac{UV}{1+U}, &  \\
V^{\prime }(t)=-V+\beta \dfrac{UV}{1+U}. &
\end{array}%
\right.  \label{2.1}
\end{equation}
The results presented in this section are mainly due to Hsu \cite{Hsu}[Theorem 3.3], Hsu, Hubbell and Waltman \cite[Lemma 4]
{Hsu-Hubbell-Waltman} where the global stability of the interior equilibrium is obtained by using the Dulac criteria, and to  Cheng \cite{Cheng81} who proved the uniqueness of the periodic orbit. In this section, we reformulate these results using the theory of the global attractor and as mentioned above we propose a refine description of this object by studying the existence and uniqueness of heteroclinic orbit starting from the no predator region $(V=0$) to the interior global attractor (where $U>0$ and $V>0$).
The results presented in the next main section, about \eqref{1.7}, will make use of the refined description presented in this section.

To study \eqref{2.1} let us first observe that this system admits the following equilibrium points.
The boundary equilibria are given by
\begin{equation*}
(\overline{U}_{0},\overline{V}_{0})=(0,0)\text{ and }(\overline{U}_{1},%
\overline{V}_{1})=(\gamma ,0).
\end{equation*}
and the unique \textit{interior equilibrium} whenever $\gamma \left( \beta -1\right) >1$, that is given by
\begin{equation*}
(\overline{U}_{2},\overline{V}_{2})=\left(\dfrac{1}{\beta -1},\dfrac{\alpha \beta %
\left[ \gamma \left( \beta -1\right) -1\right] }{\left( \beta -1\right) ^{2}}%
\right).
\end{equation*}
Next define the functions
\begin{equation*}
F(U,V) =\alpha U(\gamma -U)-\dfrac{UV}{1+U}=\frac{U}{1+U}[f(U)-V]
\end{equation*}
and
\begin{equation*}
G(U,V)=V\left( \dfrac{\beta U}{1+U}-1\right) =(\beta -1)\dfrac{V\left( U-\overline{U}_{2}\right) }{1+U}
\end{equation*}
with the nullclines
$$
f(U)=\alpha (\gamma -U)\left( 1+U\right)
$$
and
$$U=\overline{U}_{2}$$
for $U$-equation and $V$-equation, respectively. Note that the map $f(U)$ is symmetric with respect to the vertical line $U=\frac{\gamma-1}{2}$.  \\

From now on, we make use of the following assumption, ensuring that \eqref{2.1} admits the 3 equilibrium points described above.
\begin{Assumption}
\label{ASS2.1} We assume that $\gamma \left( \beta -1\right) >1$.
\end{Assumption}
Our first result investigates the local behaviour of \eqref{2.1} around the interior equilibrium.
\begin{lemma}
\label{LE2.2} Let Assumption \ref{ASS2.1} be satisfied. The interior
equilibrium is locally exponentially stable (or is a sink) if $\gamma \left(
\beta -1\right) <\beta +1$. Moreover whenever $\gamma \left( \beta -1\right)
=\beta +1$ the linearized equation has two purely imaginary (conjugated)
eigenvalues
\begin{equation*}
\lambda_\pm=\pm i \sqrt{\frac{\alpha}{\beta} \left[ \gamma \left( \beta
-1\right) -1\right]}.
\end{equation*}
Furthermore the interior equilibrium is a source if $\gamma \left( \beta
-1\right) >\beta +1$. More precisely, the linearized equation of system %
\eqref{2.1} around the interior equilibrium has two conjugated complex
eigenvalues with strictly positive real part.
\end{lemma}

\begin{remark} \label{REM2.3} System \eqref{2.1} undergoes an Hopf bifurcation around the
interior equilibrium whenever we choose the bifurcation parameter $\gamma$.
Moreover the Hopf bifurcation occurs at $\gamma=\gamma^\star$ where
\begin{equation*}
\gamma^\star \left( \beta -1\right) =\beta +1.
\end{equation*}
\end{remark}

\begin{proof}
The Jacobian matrix of the system (\ref{2.1}) at the equilibrium $(\overline{%
U}_{2},\overline{V}_{2})$ becomes
\begin{equation*}
\left(
\begin{array}{cc}
\alpha \left( \frac{\gamma \left( \beta -1\right) -\left(\beta +1 \right)}{%
\beta \left( \beta -1\right) }\right) & -\frac{1}{\beta } \\
\alpha \left[ \gamma \left( \beta -1\right) -1\right] & 0%
\end{array}%
\right) .
\end{equation*}
 and the result follows from straightforward algebra.
\end{proof}

We now discuss the existence of global attractor for \eqref{2.1}.

\begin{proposition}[Global attractor]
\label{PROP2.4} Let Assumption \ref{ASS2.1} be satisfied. There exists $R>0$
such that the triangle
\begin{equation*}
\mathbb{T}=\left\{ (U,V)\in \lbrack 0,\infty )^{2}:\beta U+V\leq R\right\}
\end{equation*}%
is positively invariant by the semiflow generated by \eqref{2.1}. The
(positive) semiflow generated by \eqref{2.1} in $\mathbb{R}^2_+$ admits a global attractor, denoted by $\mathcal{A}_{\mathbb{R}^2_+}$, that is contained in $\mathbb T$. Furthermore the triangle
$\mathbb{T}$ contains all the non negative equilibria of
\eqref{2.1}.
\end{proposition}

\begin{proof}
The result follows from the fact that for all $R>0$ large enough
\begin{equation*}
\beta F(U,V)+G(U,V)<0
\end{equation*}%
whenever $\beta U+V=R$ and $U\geq 0$ and $V\geq 0$.
\end{proof}

We now discuss the existence of an interior attractor. To that aim we consider the regions
\begin{equation*}
\partial_U \mathbb{R}^2_+= \left\lbrace (U,V) \in \mathbb{R}^2_+: V =0
\right\rbrace ,
\end{equation*}
and
\begin{equation*}
\partial_V \mathbb{R}^2_+= \left\lbrace (U,V) \in \mathbb{R}^2_+: U =0
\right\rbrace ,
\end{equation*}
as well as the interior region
\begin{equation*}
\mathrm{Int} \left( \mathbb{R}^2_+ \right) = \left\lbrace (U,V) \in \mathbb{R%
}^2_+: U >0 \text{ and } V >0 \right\rbrace.
\end{equation*}
They are all positively invariant by the semiflow generated by \eqref{2.1}.

We now decompose the state space $M:=\R^2_+$ into
the interior region
\begin{equation*}
\overset{\circ}{M}=\mathrm{Int} \left( \mathbb{R}^2_+ \right)
\end{equation*}
and the boundary region
\begin{equation*}
\partial M=M \setminus \overset{\circ}{M}= \partial_U \mathbb{R}^2_+ \cup
\partial_V \mathbb{R}^2_+.
\end{equation*}
Then by using Hale and Waltman \cite[Theorem 4.1]{Hale-Waltman} we obtain
the following uniform persistence result with respect to the state space
decomposition $(\partial M,\overset{\circ}{M})$.

\begin{proposition}[Uniform persistence]
\label{PROP2.5} Let Assumption \ref{ASS2.1} be satisfied. There exists a
constant $\Theta>0$, such that for each $(U_0,V_0) \in [0, \infty)^2$ with $%
U_0>0$ and $V_0>0$
\begin{equation*}
\liminf_{t \to \infty} U(t) \geq \Theta \text{ and } \liminf_{t \to \infty}
V(t) \geq \Theta.
\end{equation*}
\end{proposition}

\begin{proof}
Indeed the two equilibria on the boundary $M_1=\left\lbrace (0,0)
\right\rbrace $ and $M_2=\left\lbrace (\gamma,0) \right\rbrace $ are chained
in the sense of Hale and Waltman's \cite{Hale-Waltman}. Therefore it is
sufficient to prove the local repulsivity of each of these equilibria with
respect to the interior region. Assume that
\begin{equation*}
U_0>0 \text{ and } V_0>0 \text{ and } U(t)+V(t) \leq \varepsilon, \forall t
\geq 0.
\end{equation*}
Then by using the $U$-equation of \eqref{2.1} we obtain
\begin{equation*}
U^{\prime }\geq \left[ \alpha (\gamma -\varepsilon)- \varepsilon \right] U.
\end{equation*}
Therefore by choosing $\varepsilon>0$ small enough (so that $\left[ \alpha
(\gamma -\varepsilon)- \varepsilon \right]>0$) we deduce that
\begin{equation*}
\lim_{t \to \infty }U(t)= \infty
\end{equation*}
which is impossible since the system is dissipative.

Assume that
\begin{equation*}
V_0>0 \text{ and } \vert U(t)-\gamma \vert +V(t) \leq \varepsilon, \forall t
\geq 0.
\end{equation*}
Then by using the $V$-equation of \eqref{2.1} we obtain
\begin{equation*}
V^{\prime }\geq -V+\beta \dfrac{(\gamma-\varepsilon)V}{1+(\gamma-\varepsilon)%
}.
\end{equation*}
Therefore by choosing $\varepsilon>0$ small enough (so that $\beta \dfrac{%
(\gamma-\varepsilon)}{1+(\gamma-\varepsilon)}>1 \Leftrightarrow
(\beta-1)(\gamma-\varepsilon)>1 $) we deduce that
\begin{equation*}
\lim_{t \to \infty }V(t)= \infty
\end{equation*}
which is a contradiction.
\end{proof}

By using Dulac's criterion Hsu, Hubbell and Waltman \cite[Lemma 4]%
{Hsu-Hubbell-Waltman} proved that the system has no periodic orbit whenever $\gamma
\left( \beta -1\right) <\beta +1$. More precisely, setting $\varphi(U,V)=\frac{1+U}{U} V^{\xi+1}$ for some constant $\xi>0$ such that
$$
\left(\frac{\gamma-1}{2}-\frac{\gamma-1}{4}\right)\frac{4\alpha}{\beta-1}<\xi<\left(\frac{1}{\beta-1}-\frac{\gamma-1}{4}\right)\frac{4\alpha}{\beta-1},
$$
then, the aforementioned works proved that for each $0<\eta<1$ there exists $m_\eta>0$ such that
\begin{equation}\label{Dulac}
\partial_U\left(\varphi F\right)+\partial_U\left(\varphi G\right)\leq -m_\eta,\;\forall (U,V)\in \left[\eta,\eta^{-1}\right]^2.
\end{equation}
Therefore by using the Poincar\'e
Bendixson theorem we obtain the following theorem.
\begin{theorem}[Global stability]
\label{TH2.6} Let Assumption \ref{ASS2.1} be satisfied and assume that
\begin{equation*}
\gamma \left( \beta -1\right) <\beta +1.
\end{equation*}
Then the interior equilibrium is global asymptotically stable for system %
\eqref{2.1} restricted to $\mathrm{Int}(\mathbb{R}^2_+ )$.
\end{theorem}
In the following theorem the uniqueness of the periodic orbit was proved by
Cheng \cite{Cheng81} and its stability was proved by Hsu, Hubbell and
Waltman \cite{Hsu-Hubbell-Waltman}.

\begin{theorem}[Unique stable periodic orbit]
\label{TH2.8} Assume that $\gamma \left( \beta -1\right) >\beta +1.$ Then there exists a unique stable periodic orbit surrounding the interior
equilibrium and the system has no other periodic orbit.
\end{theorem}


In the following theorem we are using the notion of global attractor
considered first by Hale \cite{Hale88, Hale00}. We refer to Magal and Zhao
\cite{Magal-Zhao} and Magal \cite{Magal09} for more results and examples
about global attractors only attracting compact subsets.

\begin{proposition}
\label{PROP2.10} The semiflow generated by \eqref{2.1} restricted to $\mathbb{R}^2_+$ (respectively $\partial_U \mathbb{R}^2_+$, $%
\partial_V \mathbb{R}^2_+$ and $\mathrm{Int} \left( \mathbb{R}^2_+ \right)$)
has a global attractor $\mathcal{A}_{\mathbb{R}^2_+}$ (respectively $%
\mathcal{A}_{\partial_U \mathbb{R}^2_+}$, $\mathcal{A}_{\partial_V \mathbb{R}%
^2_+}$ and $\mathcal{A}_{\mathrm{Int} \left( \mathbb{R}^2_+ \right)}$) which
is a compact and connected subset which attracts all the compact subsets of $\mathbb{R}^2_+$ (respectively $\partial_U \mathbb{R}^2_+$, $\partial_V
\mathbb{R}^2_+$ and $\mathrm{Int} \left( \mathbb{R}^2_+ \right)$).
\end{proposition}

\begin{remark}
\label{REM2.10} The global attractor $\mathcal{A}_{\mathrm{Int} \left(
\mathbb{R}^2_+ \right)}$ only attracts the compact
subsets of $\mathrm{Int} \left( \mathbb{R}^2_+ \right)$. That is to say that
$\mathcal{A}_{\mathrm{Int} \left( \mathbb{R}^2_+ \right)}$ does not attract
the bounded subsets of the interior region $\mathrm{Int} \left( \mathbb{R}%
^2_+ \right)$ (see \cite{Magal-Zhao} for more examples).
\end{remark}

It is readily checked that the global attractor in $\partial_V \mathbb{R}%
^2_+ $ is
\begin{equation*}
\mathcal{A}_{\partial_V \mathbb{R}^2_+}=\left\lbrace (0,0) \right\rbrace
\end{equation*}
while the global attractor in $\partial_U \mathbb{R}^2_+$ is
\begin{equation*}
\mathcal{A}_{\partial_U \mathbb{R}^2_+}=\left\lbrace (U,V) \in \mathbb{R}%
^2_+: U \in [0, \gamma] \text{ and }V=0 \right\rbrace.
\end{equation*}
Indeed $\mathcal{A}_{\partial_U \mathbb{R}^2_+}$ contains the two equilibria
in $\partial_U \mathbb{R}^2_+$ as well as the heteroclinic orbit joining these two
equilibria.

\begin{proposition}
Assume that $\gamma \left( \beta -1\right) >\beta +1.$ Then the global attractor $\mathcal{A}_{\mathrm{Int} \left( \mathbb{R}^2_+
\right)}$ is the union of the periodic orbit and all the points surrounded
by the periodic orbit.
\end{proposition}

Recall that Assumption \ref{ASS2.1} is satisfied. Firstly we summarize the above result by the description of the interior attractor $\mathcal{A}_{\mathrm{Int} \left( \mathbb{R}^2_+\right)}$ depending on the parameters. The following result is a direct consequence of the above results.

\begin{theorem}[Interior attractor]
The following holds.
\begin{itemize}
\item[{\rm (i)}] Assume that
\begin{equation*}
1<\gamma \left( \beta -1\right) <\beta +1.
\end{equation*}
Then the interior attractor $\mathcal{A}_{\mathrm{Int} \left( \mathbb{R}^2_+\right)}$ reduces to the interior equilibrium.

\item[{\rm (ii)}] Assume that $\gamma \left( \beta -1\right) >\beta +1$. Then the interior attractor $\mathcal{A}_{\mathrm{Int} \left( \mathbb{R}^2_+\right)}$ consists of the unique interior equilibrium, the unique interior
periodic orbit and an infinite number of heteroclinic orbits joining the unique interior equilibrium and the unique periodic orbit.
\end{itemize}
\end{theorem}

To complete this section, we are able to describe the global attractor $\mathcal{A}_{\mathbb{R}^2_+}$.
Our result reads as follows.
\begin{theorem}[Global attractor]
System \eqref{2.1} admits a unique heteroclinic orbit $(U,V)$ joining $(\gamma,0)$ to the boundary of the interior attractor $\mathcal{A}_{\mathrm{Int} \left( \mathbb{R}^2_+
\right)}$. The global attractor $\mathcal{A}_{\mathbb{R}^2_+}$ is composed of 3 disjoint parts
\begin{equation*}
\mathcal{A}_{\mathbb{R}^2_+}=[0,\gamma]\times\{0\} \bigcup \left\{(U(t),V(t)),\;t\in\R\right\}\bigcup \mathcal{A}_{\mathrm{Int} \left( \mathbb{R}^2_+
\right)}.
\end{equation*}
\end{theorem}

\begin{proof}
The proof of this result requires three steps. We firstly derive the existence of heteroclinic orbits using a connectedness argument for the global attractor. Then we show that heteroclinic orbits starts from the stationary point $(\gamma,0)$ and finally we conclude to the uniqueness of such heroclinic orbit by using a center unstable manifold argument (see \cite{Ducrot-Langlais-Magal} where a rather similar argument was used to derive a uniqueness property for traveling wave solutions arising in some epidemic problem).

\noindent \textbf{Connectedness arguments:} The largest global attractor  $\mathcal{A}_{\mathbb{R}^2_+}$  is connected since it attracts the convex subset $\mathbb{T}$. Since any continuous map maps a connected set into a connected set, it follows that the projection of $\mathcal{A}_{\mathbb{R}^2_+}$ on the horizontal and vertical axis is a compact interval (since a one dimensional compact connected set is an compact interval).

The global attractor $\mathcal{A}_{\mathbb{R}^2_+}$  contains the interior global attractor  $\mathcal{A}_{\mathrm{Int} \left( \mathbb{R}^2_+\right)}$ which is compact, connected and locally stable. The global attractor $\mathcal{A}_{\mathbb{R}^2_+}$ also contains the boundary attractor $\mathcal{A}_{\partial_U \mathbb{R}^2_+}$. The connectedness of $\mathcal{A}_{\mathbb{R}^2_+} $  and compactness  of $\mathcal{A}_{\mathrm{Int} \left( \mathbb{R}^2_+\right)}$ and $\mathcal{A}_{\partial_U \mathbb{R}^2_+} $ imply
$$
\mathcal{A}_{\mathbb{R}^2_+} - \left(\mathcal{A}_{\mathrm{Int} \left( \mathbb{R}^2_+\right)} \bigcup  \mathcal{A}_{\partial_U \mathbb{R}^2_+} \right) \neq \emptyset.
$$
Moreover by using Theorem 3.2 due to Hale and Waltman \cite{Hale-Waltman} we deduce that for each point $(U,V) \in \mathcal{A}_{\mathbb{R}^2_+} - \left(\mathcal{A}_{\mathrm{Int} \left( \mathbb{R}^2_+\right)} \bigcup  \mathcal{A}_{\partial_U \mathbb{R}^2_+} \right)$
the alpha and limit sets  satisfy the following
$$
\alpha(U,V) \in \mathcal{A}_{\partial_U \mathbb{R}^2_+} \text{ and } \omega(U,V) \in \mathcal{A}_{\mathrm{Int} \left( \mathbb{R}^2_+\right)}.
$$
Finally since the boundary attractor has a Morse decomposition $M_1=\left\lbrace (0,0) \right \rbrace$ and $M_2=\left\lbrace (\gamma,0) \right \rbrace$ we have either
$$
\alpha(U,V) =M_1 \text{ or } \alpha(U,V) =M_2, \forall (U,V) \in \mathcal{A}_{\mathbb{R}^2_+} - \left(\mathcal{A}_{\mathrm{Int} \left( \mathbb{R}^2_+\right)} \bigcup  \mathcal{A}_{\partial_U \mathbb{R}^2_+} \right).
$$

\noindent \textbf{No existence of heteroclinic orbit starting from $(0,0)$:  }Assume by contradiction that there exists one. By looking the $V$-equation
$$
 V^{\prime }=\left( -1+ \dfrac{U}{1+U} \right)V,
$$
we deduce that
$$
V(t)=\exp \left( \int_{t_0}^t  -1+ \dfrac{U(s)}{1+U(s)}  ds \right) V(t_0).
$$
Since $V(t_0)>0$ and there exists $T<0$ such that $U(t)$ remains sufficiently small for all negative times $t<T$, we deduce that
$$
\lim_{t \to -\infty}V(t)=+\infty
$$
which contradicts the fact that the solution belongs to the global attractor and is therefore bounded.

\bigskip
\noindent \textbf{Existence and uniqueness of an heteroclinic orbit starting from $(\gamma,0)$:  }
We only need to prove the uniqueness. The linearized equation around $(\gamma,0)$ has two eigenvalues: $\lambda _{1}=-\alpha \gamma<0$ and $\lambda _{2}=\frac{\beta \gamma}{1+\gamma}-1>0$. The center-unstable manifold at $(\gamma ,0)$
is one dimensional. Note that%
\begin{equation*}
E_{\lambda _{1}}=\left\{ \left( U,V\right) \in \mathbb{R}^{2}:V=0\right\}
\end{equation*}%
and%
\begin{equation*}
E_{\lambda _{2}}=\left\{ \left( U,V\right) \in \mathbb{R}^{2}:U-\gamma =-%
\frac{\gamma }{\gamma \left( \beta -1\right) -1+\alpha \,\gamma \left(
1+\gamma \right) }V\right\}
\end{equation*}%
with $\frac{\gamma }{\gamma \left( \beta -1\right) -1+\alpha \,\gamma \left(
1+\gamma \right) }>0.$ Note that $\mathbb{R}^{2}=E_{\lambda _{1}}\bigoplus E_{\lambda _{2}}.$ \\
Let $\psi _{cu}:$ $E_{\lambda _{2}}\rightarrow
E_{\lambda _{1}}$ be a $C^{1}$ center-unstable manifold and consider the one
dimensional manifold defined by%
\begin{equation*}
M_{cu}:=\{x_{cu}+\psi _{cu}(x_{cu}):x_{cu}\in E_{\lambda _{2}}\}.
\end{equation*}%
It is locally invariant under the semiflow generated by \eqref{2.1} around $%
(\gamma ,0)$. Since $D_{x_{cu}}\psi _{cu}(0)=0,$ the manifold $M_{cu}$ is
tangent to $E_{\lambda _{2}}$ at $(\gamma ,0)$. Moreover we know that there
exists $\varepsilon >0$, such that $M_{cu}$ contains all negative orbits of
the semiflow generated by \eqref{2.1} staying in the ball $B_{\mathbb{R}^{2}}((\gamma ,0),\varepsilon)$
for all negative times.

In order to prove the uniqueness, we assume that
there exists two heteroclinic orbits
$$O_{1} = {(U_{1}(t), V_{1}(t))}_{t\in \mathbb{R}}\subset {\mathrm{Int} \left( \mathbb{R}^2_+\right)}$$
and
$$O_{2} = {(U_{2}(t), V_{2}(t))}_{t\in \mathbb{R}}\subset {\mathrm{Int} \left( \mathbb{R}^2_+\right)}$$
going from $(\gamma,0)$ to the interior attractor $\mathcal{A}_{\mathrm{Int} \left( \mathbb{R}^2_+\right)}$.
Since
$$\lim_{t\rightarrow -\infty}(U_{1}(t), V_{1}(t))=(\gamma,0)$$
and
$$\lim_{t\rightarrow -\infty}(U_{2}(t), V_{2}(t))=(\gamma,0),$$
without loss of generality, one may assume that
$${(U_{1}(t), V_{1}(t))}_{t\leq 0}\subset B_{\mathbb{R}^{2}}((\gamma ,0),\varepsilon )$$
and
$${(U_{2}(t), V_{2}(t))}_{t\leq 0}\subset B_{\mathbb{R}^{2}}((\gamma ,0),\varepsilon )$$
which imply that
$${(U_{1}(t), V_{1}(t))}_{t\leq 0}\subset M_{cu}$$
and
$${(U_{2}(t), V_{2}(t))}_{t\leq 0}\subset M_{cu}.$$
Let $\Pi_{\lambda_{1}}$ and $\Pi_{\lambda_{2}}$ be the linear projectors from $\mathbb{R}^{2}$ to $E_{\lambda _{1}}$ and $E_{\lambda _{2}}$, respectively.
We can find $t_{1} < 0$ and $t_{2} < 0$ such that $\Pi_{\lambda_{2}}(U_{1}(t_{1}), V_{1}(t_{1})) =\Pi_{\lambda_{2}}(U_{2}(t_{2}), V_{2}(t_{2}))$ and then $\psi _{cu}(\Pi_{\lambda_{2}}(U_{1}(t_{1}), V_{1}(t_{1})))=\psi _{cu}(\Pi_{\lambda_{2}}(U_{2}(t_{2}), V_{2}(t_{2})))$. Thus $(U_{1}(t_{1}), V_{1}(t_{1}))=(U_{2}(t_{2}), V_{2}(t_{2}))$.
By the uniqueness of the solutions for system \eqref{2.1}, we get $(U_{1}(t_{1}+\cdot), V_{1}(t_{1}+\cdot))=(U_{2}(t_{2}+\cdot), V_{2}(t_{2}+\cdot))$ and thus $O_{1}= O_{2}$.  The uniqueness of the heteroclinic orbit starting from $(\gamma,0)$ follows and this completes the proof of the theorem.

\end{proof}

\begin{figure}[H]
		\centering
		\textbf{(a)}\\
		\includegraphics[width=5in,height=3in]{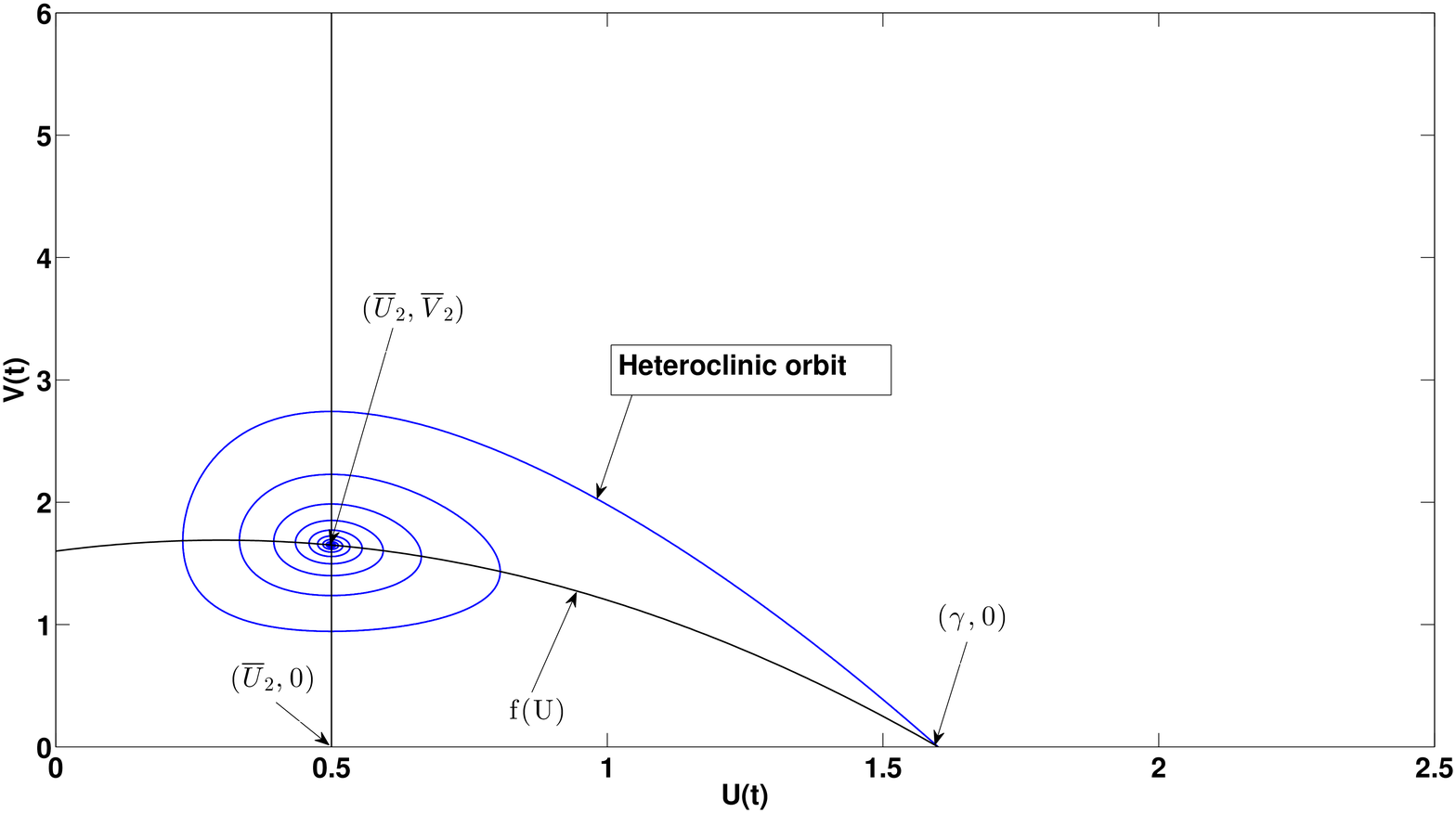}\\
		\textbf{(b)}\\
		\includegraphics[width=5in,height=3in]{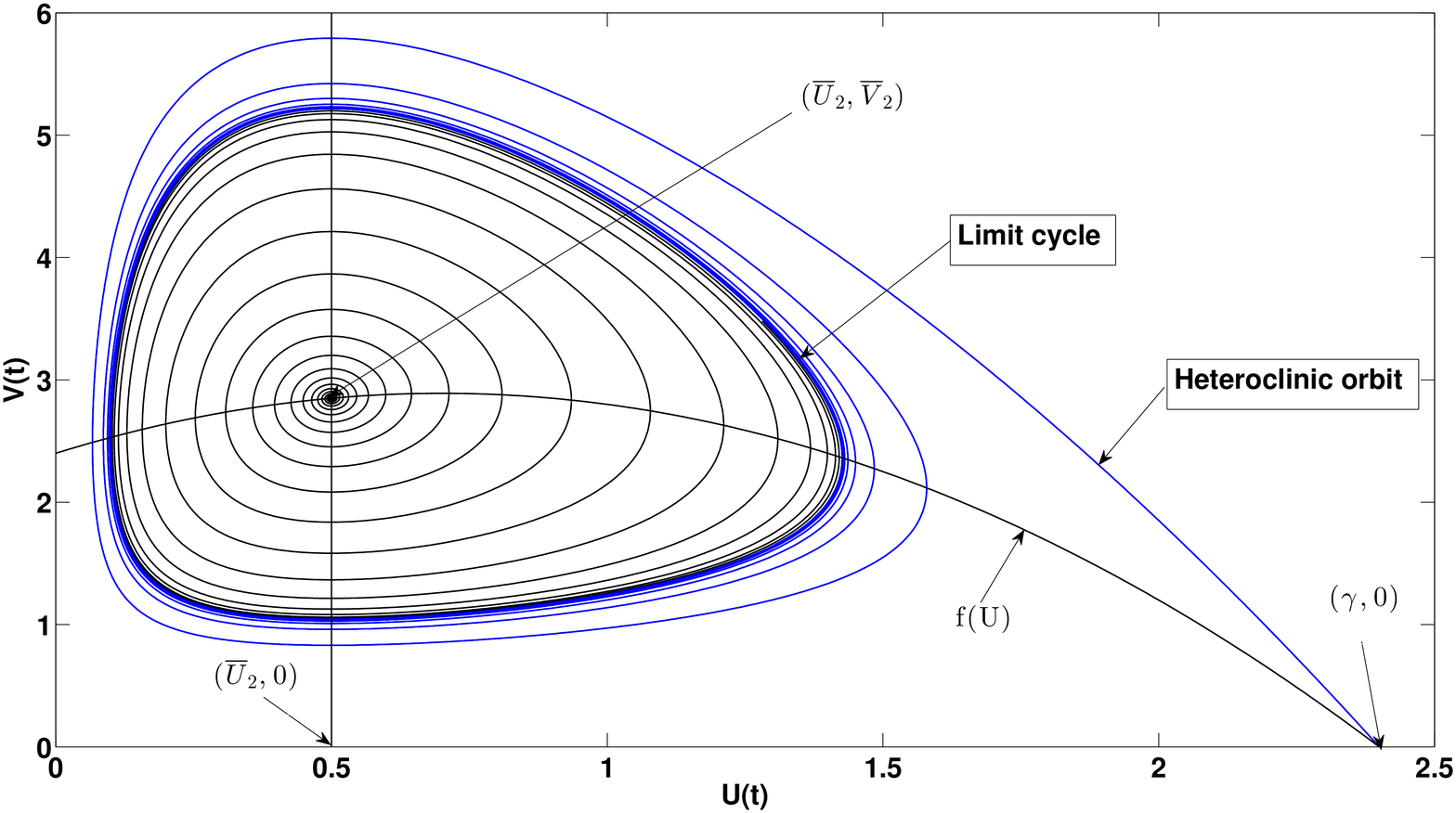}
		\caption{\textit{In  this figure we run a simulation of the  Rosenzweig-MacArthur model with $\alpha=1$, $\beta=3$ and $\gamma=1.6$ (in Figure (a)) and $\gamma=2.4$ (in Figure (b)). In both figures we plot the heteroclinic orbit joining the boundary equilibrium and the interior equilibrium (in Figure (a)) and the interior limit cycle which is a stable periodic orbit (in Figure (b)). In this figure we also plot the nullclines $f(U)=\alpha (\gamma -U)\left( 1+U\right) $ and $U=\overline{U}_2$. }}\label{fig1}
\end{figure}	

\section{Application of a center manifold theorem to the traveling wave problem}

This section is devoted to the study of traveling wave profile system of equations \eqref{1.7} for $\varepsilon\ll 1$.  We will firstly apply a center manifold reduction on a suitable invariant region. The reduced system will be analysed. In the same spirit as in the previous section we will describe its global and interior attractor to obtain various results about the existence and uniqueness of traveling wave solutions as well as refined information about periodic wave trains.

\subsection{Reduction of the traveling wave problem}

\noindent \textbf{Transformed system:} In order to work with a subspace of
equilibria for $\varepsilon=0$ we use the following change of variable
\begin{equation}  \label{3.1}
\left\lbrace
\begin{array}{lll}
U_1 & = & u_1, \\
U_2 & = & u_2+F(u_1,v_1), \\
V_1 & = & v_1, \\
V_2 & = & v_2+G(u_1,v_1)%
\end{array}
\right. \Leftrightarrow \left\lbrace
\begin{array}{lll}
u_1 & = & U_1, \\
u_2 & = & U_2-F(U_1,V_1), \\
v_1 & = & V_1, \\
v_2 & = & V_2 -G(U_1,V_1).%
\end{array}
\right.
\end{equation}
By using this change of variable the system \eqref{1.7} becomes
\begin{equation*}
\left\lbrace
\begin{array}{lll}
U_1^{\prime } & = & \varepsilon u_2 =\varepsilon \left[U_2-F(U_1,V_1) \right],
\\
d U_2^{\prime } & = & d u_2^{\prime }+d F(U_1,V_1)^{\prime }=-U_2+d
\partial_u F(U_1,V_1) U_1^{\prime }+ d\partial_v F(U_1,V_1) V_1^{\prime }, \\
V_1^{\prime } & = & \varepsilon v_2=\varepsilon \left[V_2-G(U_1,V_1) \right],
\\
V_2^{\prime } & = & v_2^{\prime }+G(U_1,V_1)^{\prime }= -V_2 +\partial_u
G(U_1,V_1) U_1^{\prime }+\partial_v G(U_1,V_1) V_1^{\prime }%
\end{array}
\right.
\end{equation*}
and therefore we obtain

\begin{equation}  \label{3.2}
\left\lbrace
\begin{array}{lll}
U_1^{\prime } & = & \varepsilon \left[U_2-F(U_1,V_1) \right], \\
dU_2^{\prime } & = & -U_2+\varepsilon d P(U_1,U_2,V_1,V_2), \\
V_1^{\prime } & = & \varepsilon \left[V_2-G(U_1,V_1) \right], \\
V_2^{\prime } & = & -V_2+\varepsilon Q(U_1,U_2,V_1,V_2),%
\end{array}
\right.
\end{equation}
wherein $P$ and $Q$ are given by
\begin{equation*}
P(U_1,U_2,V_1,V_2)=\partial_u F(U_1,V_1) \left[U_2-F(U_1,V_1) \right]%
+\partial_v F(U_1,V_1)\left[V_2-G(U_1,V_1) \right]
\end{equation*}
and
\begin{equation*}
Q(U_1,U_2,V_1,V_2)=\partial_u G(U_1,V_1) \left[U_2-F(U_1,V_1) \right]
+\partial_v G(U_1,V_1) \left[V_2-G(U_1,V_1) \right].
\end{equation*}

\bigskip

\noindent \textbf{Truncated system:} Let $\rho :\mathbb{R} \to \mathbb{R}$ be a $%
C^\infty$ function such that
\begin{equation*}
\rho(x) = \left\lbrace
\begin{array}{l}
1, \text{ if } x \geq 0, \\
\in [0,1], \text{ if } x \in [-1/2,0], \\
0, \text{ if } x \leq -1/2.%
\end{array}
\right.
\end{equation*}
Define the set
\begin{equation*}
\mathbb{E}=\left\lbrace (U_1,U_2,V_1,V_2) \in \mathbb{R}^4: (U_1,V_1) \in
\mathbb{T} \text{ and } \vert U_2-F(U_1,V_1) \vert \leq 1, \vert
V_2-G(U_1,V_1) \vert \leq 1 \right\rbrace.
\end{equation*}
Let $L>0$ be given large enough such that
\begin{equation*}
\begin{array}{ll}
L  \geq & 2+\displaystyle \max_{(U_1,V_1) \in \mathbb{T}} \vert F(U_1,V_1)
\vert+ \max_{(U_1,V_1) \in \mathbb{T}} \vert G(U_1,V_1) \vert \\
& +\displaystyle \max_{(U_1,U_2,V_1,V_2) \in \mathbb{E} } \vert
P(U_1,U_2,V_1,V_2) \vert \\
& +\displaystyle \max_{(U_1,U_2,V_1,V_2) \in \mathbb{E}} \vert
Q(U_1,U_2,V_1,V_2) \vert.%
\end{array}%
\end{equation*}
Let $\chi :\mathbb{R} \to \mathbb{R}$ be a $C^\infty$ function such that
\begin{equation*}
\chi(x) = \left\lbrace
\begin{array}{l}
1, \text{ if } x \in [-L,L], \\
\in [0,1], \text{ if } x \in [-(L+1),-L] \cup [L,(L+1)], \\
0, \text{ if } x \notin [-(L+1),(L+1)].%
\end{array}
\right.
\end{equation*}
Then we have
\begin{equation}
\left\lbrace \label{3.3}
\begin{array}{lll}
U_1^{\prime } & = & \varepsilon \left[U_2-F(U_1,V_1) \right]
\chi(U_2-F(U_1,V_1) ) \rho(U_1), \\
dU_2^{\prime } & = & -U_2+\varepsilon d P(U_1,U_2,V_1,V_2)
\chi(P(U_1,U_2,V_1,V_2) ) \rho(U_1), \\
V_1^{\prime } & = & \varepsilon \left[V_2-G(U_1,V_1) \right]
\chi(V_2-G(U_1,V_1) ) \rho(U_1), \\
V_2^{\prime } & = & -V_2+\varepsilon Q(U_1,U_2,V_1,V_2)
\chi(Q(U_1,U_2,V_1,V_2) ) \rho(U_1).%
\end{array}
\right.
\end{equation}
Define
\begin{equation*}
h(x)=x \chi(x),\;x\in\mathbb R.
\end{equation*}
Then system \eqref{3.3} can be rewritten as
\begin{equation}
\left\lbrace \label{3.4}
\begin{array}{lll}
U_1^{\prime } & = & \varepsilon \, h\left( U_2-F(U_1,V_1) \right) \rho(U_1),
\\
dU_2^{\prime } & = & -U_2+\varepsilon \,d \, h \left( P(U_1,U_2,V_1,V_2)
\right) \rho(U_1), \\
V_1^{\prime } & = & \varepsilon \, h \left(V_2-G(U_1,V_1) \right) \rho(U_1),
\\
V_2^{\prime } & = & -V_2+\varepsilon \, h \left( Q(U_1,U_2,V_1,V_2) \right)
\rho(U_1).%
\end{array}
\right.
\end{equation}
\begin{remark} In this truncation procedure the function $\rho(U_1)$ serves to avoid the singularity at $U_1=-1$ in $F$ and $G$. The function $h(.)$ is used to obtain a bounded Lipschitz perturbation of the system with $\varepsilon=0$.  
\end{remark}
By setting $X(t)=(U_1(t),V_1(t))$ and $Y(t)=(U_2(t),V_2(t))$, system \eqref{3.4} takes the following form
\begin{equation}  \label{3.5}
\left\lbrace
\begin{array}{l}
X^{\prime }(t)=\varepsilon \widetilde{F}(X(t),Y(t)), \\
Y^{\prime }(t)=-DY(t)+\varepsilon \widetilde{G}(X(t),Y(t)),%
\end{array}
\right.
\end{equation}
where $\widetilde{F}, \widetilde{G} \in C^\infty \left(\mathbb{R}^2 \times
\mathbb{R}^2 , \mathbb{R}^2 \right)$ are bounded and Lipschitz continuous functions
and where we have set $D=\mathrm{diag}(d^{-1},1)$. Therefore the central space is given by
\begin{equation*}
X_c=\left\lbrace (X,Y) \in \mathbb{R}^2 \times \mathbb{R}^2:Y=0
\right\rbrace,
\end{equation*}
while the stable space reads as
\begin{equation*}
X_s=\left\lbrace (X,Y) \in \mathbb{R}^2 \times \mathbb{R}^2 :X=0
\right\rbrace.
\end{equation*}

\begin{remark}
\label{REM3.1} Due to the choice of the constant $L>0$ the truncated system %
\eqref{3.4} coincides with the original system \eqref{3.2} whenever $%
(U_1,U_2,V_1,V_2) \in \mathbb{E}$. Moreover the equilibria of system %
\eqref{3.2} belong to $\mathbb{E}$ (since $U_1=u_1$ and $V_1=v_1$ and the
equilibria of \eqref{1.7} satisfy $u_2=v_2=0$ and $(u_1,v_1)$ must be an
equilibrium of \eqref{2.1}). Conversely the equilibria of \eqref{3.4} in $%
\mathbb{E}$ must satisfy
\begin{equation*}
\widetilde{U}_2=F(\widetilde{U}_1,\widetilde{V}_1)=0 \text{ and }\widetilde{V}%
_2=G(\widetilde{U}_1,\widetilde{V}_1)=0.
\end{equation*}
Now by using Proposition \ref{PROP2.4} we have $(\widetilde{U}_1,%
\widetilde{V}_1) \in \mathbb{T}$.
\end{remark}

For $\eta > 0$ and $p \in \mathbb N\setminus\{0\}$ we define the weighted spaces
\begin{equation*}
BC^\eta(\mathbb{R};\mathbb{R}^p)=\left\lbrace u \in C\left(\mathbb{R},%
\mathbb{R}^p \right):\sup_{ t \in \mathbb{R}} e^{-\eta \vert t \vert} \Vert
u(t) \Vert < \infty \right\rbrace.
\end{equation*}
Moreover for $\varepsilon >0$ small enough we can apply the smooth center
manifold theorem proved by Vanderbauwhede \cite[Theorem 3.1]{Vanderbauwhede}
and Vanderbauwhede and Iooss \cite[Theorem 1]{Vanderbauwhede-Iooss}. This yields the following reduction result.

\begin{theorem} \label{TH3.2}
Let $\eta \in \left(0,\min(1,1/d)\right)$ be given and fixed. Then there exists $\widetilde\varepsilon_0 >0$
such that for each $\varepsilon \in [0,\widetilde\varepsilon_0]$ we can find a map $%
\Phi_\varepsilon=\left(\Phi^1_\varepsilon,\Phi^2_\varepsilon\right) \in C^k(\mathbb{R}^2 ,\mathbb{R}^2 )$, for each integer $%
k>0 $, satisfying the following properties
\begin{equation*}
\Phi_\varepsilon\left( 0_\mathbb{R}\right) =0_{\mathbb{R}^2} \text{ and } D
\Phi_\varepsilon\left( 0_\mathbb{R}\right) =0_{\mathcal{L} (\mathbb{R}^2)},
\end{equation*}
and $\Phi_\varepsilon$ is bounded as well as its derivatives up to the order
$k$ and
\begin{equation*}
\lim_{\varepsilon \to 0 }\Vert \Phi_\varepsilon \Vert_{\infty}=0 \text{ and }
\lim_{\varepsilon \to 0 }\Vert \Phi_\varepsilon \Vert_{\mathrm{Lip}}=0.
\end{equation*}
Moreover we have the following properties:

\begin{itemize}
\item[{\rm (i) }] The \textbf{global center manifold } $%
M_\varepsilon=\left\lbrace (X,Y): Y= \Phi_\varepsilon(X)\right\rbrace$ is
invariant by the semiflow generated by \eqref{3.4} (forward and backward in
time). Namely if $t \to X(t)$ is a solution of the \textbf{reduced
system} on some interval $I \subset \mathbb{R}$
\begin{equation}
\label{3.6}
X^{\prime }(t)=\varepsilon \widetilde{F}(X(t),\Phi_\varepsilon(X(t))),
\forall t \in I,
\end{equation}
then $t \to (X(t),\Phi_\varepsilon(X(t)))$ is a solution of \eqref{3.4} on $%
I $.

\item[{\rm(ii) }] If $t \to (X(t),Y(t))$ is a solution of \eqref{3.4} on
$\mathbb{R}$ which belongs to $BC^\eta(\mathbb{R};\mathbb{R}^4)$, then
\begin{equation*}
(X(t),Y(t)) \in M_\varepsilon, \forall t \in \mathbb{R} \Leftrightarrow
Y(t)=\Phi_\varepsilon(X(t)), \forall t \in \mathbb{R} .
\end{equation*}
\end{itemize}
\end{theorem}
Now let us prove the following invariance result.
\begin{proposition}
There exists $\varepsilon_0 \in (0,\widetilde\varepsilon_0]$ such that triangle $%
\mathbb{T}$ is negatively invariant by the flow generated by the reduced
system \eqref{3.6}. That is to say that
\begin{equation*}
X^{\prime }(t)=\varepsilon \widetilde{F}(X(t),\Phi_\varepsilon(X(t))),
\forall t \in \mathbb{R} \text{ and } X(0)=X_0 \in \mathbb{T}
\end{equation*}
implies
\begin{equation*}
X(t) \in \mathbb{T}, \forall t \leq 0.
\end{equation*}
Furthermore the subsets
\begin{equation*}
\partial_u \mathbb{T}=\left\lbrace (U,V) \in \mathbb{T}:V=0 \right\rbrace
\end{equation*}
and
\begin{equation*}
\partial_v \mathbb{T}=\left\lbrace (U,V) \in \mathbb{T}:U=0 \right\rbrace
\end{equation*}
are negatively invariant by the flow generated by the reduced system %
\eqref{3.6}.
\end{proposition}

\begin{proof}

In the first step, we investigate the invariance for the boundary regions $\partial_u\mathbb T$ and $\partial_v\mathbb T$.
To that aim we claim that
\begin{equation}  \label{3.7}
\left(
\begin{array}{c}
U_2 \\
V_2%
\end{array}
\right)= \Phi_\varepsilon\left(
\begin{array}{c}
0 \\
V_1%
\end{array}
\right) \Rightarrow U_2=0.
\end{equation}
Indeed, assume that $U_1=U_2=0$ in system, then
\begin{equation*}
U_1=U_2=0 \Rightarrow U_2-F(U_1,V_1)=0 \text{ and } P(U_1,U_2,V_1,V_2)=0.
\end{equation*}
Therefore the two last components of the truncated system \eqref{3.3} become
\begin{equation}  \label{3.8}
\left\lbrace
\begin{array}{lll}
V_1^{\prime } & = & \varepsilon \, h \left(V_2-G(0,V_1) \right), \\
V_2^{\prime } & = & -V_2+\varepsilon \, h \left( Q(0,0,V_1,V_2) \right).%
\end{array}
\right.
\end{equation}
Now by applying the center manifold theorem to \eqref{3.8} (which applies for the
value of $\varepsilon \in (0, \widetilde\varepsilon_0 )$ since the estimations for
systems \eqref{3.4} and \eqref{3.8} remain unchanged in the proof of the center
manifold theorem), we deduce that we can find a map $\Psi_\varepsilon \in
C^k(\mathbb{R},\mathbb{R})$ such that the center manifold of the two
dimensional system \eqref{3.8}
\begin{equation*}
V_2= \Psi_\varepsilon(V_1)
\end{equation*}
and the solution $t \to (V^\star_1(t),V^\star_2(t))$ of \eqref{3.8} starting
from an initial value $(V_1,\Psi_\varepsilon(V_1))$ satisfies
\begin{equation*}
(V^\star_1,V^\star_2) \in BC^\eta(\mathbb{R};\mathbb{R}^2).
\end{equation*}
We conclude that
\begin{equation*}
(U_1,U_2,V_1,V_2)=(0,0,V^\star_1,V^\star_2) \in BC^\eta(\mathbb{R};\mathbb{R}%
^4)
\end{equation*}
is a solution of the truncated system \eqref{3.4}. This completes the proof of
the claim.

By using similar argument one deduces that
\begin{equation}  \label{3.9}
\left(
\begin{array}{c}
U_2 \\
V_2%
\end{array}
\right)= \Phi_\varepsilon\left(
\begin{array}{c}
U_1 \\
0%
\end{array}
\right) \text{ and } U_1 \geq 0 \Rightarrow V_2=0.
\end{equation}
We now turn to the invariance of the triangle $\mathbb T$.
By using the fact that
\begin{equation}  \label{3.10}
\left\lbrace
\begin{array}{lll}
U_1^{\prime } & = & \varepsilon \, h\left(
\Phi^1_\varepsilon(U_1,V_1)-F(U_1,V_1) \right) \rho(U_1), \\
V_1^{\prime } & = & \varepsilon \, h
\left(\Phi^2_\varepsilon(U_1,V_1)-G(U_1,V_1) \right) \rho(U_1).%
\end{array}
\right.
\end{equation}
Whenever $\beta U_1+V_1=R$ and $U_1 \geq 0$ and $V_1 \geq 0$ in system %
\eqref{3.9}, then $\rho(U_1)=1$ and for $\varepsilon >0$ small enough ($%
h $ coincides with identity)
\begin{equation*}
\beta U_1^{\prime }+V_1^{\prime }= \varepsilon \, \left(
\Phi^1_\varepsilon(U_1,V_1)+\Phi^2_\varepsilon(U_1,V_1)-F(U_1,V_1)
-G(U_1,V_1)\right)>0.
\end{equation*}
Therefore by combining this fact together with \eqref{3.6} and \eqref{3.8},
we deduce that the triangle $\mathbb{T}$ is negatively invariant by the
reduced system.
\end{proof}

\subsection{Global attractors}

We investigate preliminary properties of the perturbed two-dimensional
(reduced) system \eqref{3.6}. Recall that Assumption \ref{ASS2.1} is
satisfied along this paper and $\mathbb{T}$, $\partial_u\mathbb{T}$, $\partial_v\mathbb{T}$ are
negatively invariant with respect to this system for all $\varepsilon\in
(0,\varepsilon_0]$. Before going further, by setting $t=-\varepsilon s$ and $%
\left(\tilde U,\tilde V\right)(s)=\left(U_1,U_2\right)(t)$ the above system \eqref{3.6} becomes, dropping the tilde for notational simplicity
\begin{equation}  \label{3.11}
\left\lbrace
\begin{array}{lll}
U^{\prime } & = & \left[ -\Phi^1_\varepsilon(U,V)+F(U,V) \right]%
:=F_\varepsilon(U,V), \\
V^{\prime } & = & \left[ -\Phi^2_\varepsilon(U,V)+G(U,V) \right]%
:=G_\varepsilon(U,V).%
\end{array}
\right.
\end{equation}
Notice that $\mathbb{T}$, $\partial_u\mathbb{T}$ and $\partial_v\mathbb{T}$
become positively invariant with respect to the above system. Then, for
each such $\varepsilon\in [0,\varepsilon_0]$, we denote by $%
\left\{T_\varepsilon(t)\right\}_{t\geq 0}$ the
strongly continuous semiflow on the triangle $\mathbb{T}$ generated by \eqref{3.11}. One may also observe
it continuously depends on $\varepsilon$, namely the map $(\varepsilon, t,
X)\to T_\varepsilon(t)X$ is continuous from $[0,\varepsilon_0]\times
[0,\infty)\times\mathbb{T}$ into $\mathbb{T}$. Our first result reads as
follows:

\begin{lemma}
\label{LE3.4} Let $\varepsilon\in [0,\varepsilon_0]$ be given. Then the
semiflow $\{T_\varepsilon(t)\}_{t\geq 0}$ possesses a compact and connected
global attractor $\mathcal{A}_\varepsilon\subset \mathbb{T}$ attracting $%
\mathbb{T}$ in the sense that
\begin{equation*}
\mathrm{dist}\,\left(T_\varepsilon(t)X,\mathcal{A}_\varepsilon\right) \to 0%
\text{ as $t\to\infty$ uniformly for $X\in\mathbb{T}$},
\end{equation*}
wherein $\mathrm{dist}\,(X,\mathcal{A}_\varepsilon)=\displaystyle \inf_{Y \in \mathcal{A}_\varepsilon} \Vert X-Y \Vert $ denotes the Euclidean
distance from $X\in\mathbb{R}^2$ to $\mathcal{A}_\varepsilon$.
\end{lemma}

\begin{proof}
Fix $\varepsilon\in (0,\varepsilon_0]$. Since, for each $t\geq 0$, $%
T_\varepsilon(t):\mathbb{T}\to \mathbb{T}$ is completely continuous and
bounded dissipative ($\mathbb{T}$ is compact), Theorem 3.4.8 in \cite{Hale00}
ensures the existence of a global attractor for the semiflow $T_\varepsilon$%
. In addition, since $\mathbb{T}$ is connected, the result of Gobbino and
Sardella \cite{Gobbino97} applies and ensures that $\mathcal{A}_\varepsilon$
is connected.
\end{proof}

\begin{lemma}
\label{LE3.5} The family $\left(\mathcal{A}_\varepsilon\right)_{\varepsilon%
\in [0,\varepsilon_0]}$ is upper semi-continuous, in the sense that for each
$\widehat\varepsilon\in [0,\varepsilon_0]$ one has
\begin{equation*}
\lim_{\varepsilon\to\widehat\varepsilon}\delta\left(\mathcal{A}_\varepsilon,%
\mathcal{A}_{\widehat\varepsilon}\right)=0,
\end{equation*}
wherein $\delta\left(\mathcal{A}_\varepsilon,\mathcal{A}_{\widehat%
\varepsilon}\right)$ is given by
\begin{equation*}
\delta\left(\mathcal{A}_\varepsilon,\mathcal{A}_{\widehat\varepsilon}%
\right)=\sup_{y\in \mathcal{A}_\varepsilon}\mathrm{dist}\,\left(y,\mathcal{A}%
_{\widehat\varepsilon}\right).
\end{equation*}
\end{lemma}

\begin{proof}
Since the map $(\varepsilon,t,X)\mapsto T_\varepsilon(t)X$ is continuous
from $[0,\varepsilon_0]\times[0,\infty)\times \mathbb{T}$ into the compact
set $\mathbb{T}$, Theorem 3.5.2 in \cite{Hale00} ensures that the family $%
\{A_\varepsilon\}_{\varepsilon\in [0,\varepsilon_0]}$ is upper
semi-continuous.
\end{proof}

We continue this section by further studying some properties of the global
attractor $\mathcal{A}_\varepsilon$. To that aim, we define
\begin{equation*}
\partial\mathbb{T}^0=\partial_u\mathbb{T}\cup \partial_v\mathbb{T}\text{ and
}\mathbb{T}^0=\mathbb{T}\setminus\partial\mathbb{T}^{0}=\left\{(U,V)\in\mathbb{T}%
:\;U>0\text{ and }V>0\right\}.
\end{equation*}
Here let us recall that, for all $\varepsilon\in [0,\varepsilon_0]$ and $%
t\geq 0$, one has
\begin{equation}  \label{3.12}
T_\varepsilon(t)\mathbb{T}^0\subset \mathbb{T}^0\text{ and }%
T_\varepsilon(t)\partial \mathbb{T}^0\subset\partial\mathbb{T}^0.
\end{equation}

We prove the following uniform persistence result for $T_\varepsilon$.

\begin{lemma}
\label{LE3.6} There exists $\varepsilon_1\in \left(0,\varepsilon_0\right]$
and $\Theta>0$ such that for all $\varepsilon\in \left(0,\varepsilon_1\right]
$ and $X\in\mathbb{T}^0$ one has
\begin{equation*}
\liminf_{t\to \infty} \mathrm{dist}\,\left(T_\varepsilon(t)X,\partial\mathbb{%
T}^{0}\right)\geq \Theta.
\end{equation*}
\end{lemma}

The proof of this lemma relies on the application of the results of Hale and
Waltman in \cite{Hale-Waltman}.

\begin{proof}
Firstly recall that
\begin{equation*}
\left(F_\varepsilon,G_\varepsilon\right)\to \left(F,G\right)\text{ as $%
\varepsilon\to 0$ in $C^1(\mathbb{T})$}.
\end{equation*}
Next fix $\varepsilon_1\in (0,\varepsilon_0]$ such that
\begin{equation}  \label{3.13}
\begin{split}
&\partial_U F_\varepsilon(0,0)>\frac{1}{2}\partial_U F(0,0)>0,\;\;\partial_U
F_\varepsilon(\gamma,0)<\frac{1}{2}\partial_U F(\gamma,0)<0, \\
&\partial_V G_\varepsilon(0,0)<\frac{1}{2}\partial_V G(0,0)<0,\;\;\partial_V
G_\varepsilon(\gamma,0)>\frac{1}{2}\partial_V G(\gamma,0)>0.
\end{split}%
\end{equation}
Now, in order to apply the result of Hale and Waltman, consider the extended
semiflow $U(t):\mathbb{T}\times [0,\varepsilon_1]\to \mathbb{T}\times
[0,\varepsilon_1]$ given by
\begin{equation*}
U(t)%
\begin{pmatrix}
X \\
\varepsilon%
\end{pmatrix}%
:=%
\begin{pmatrix}
T_\varepsilon(t)X \\
\varepsilon%
\end{pmatrix}%
,\;\forall
\begin{pmatrix}
X \\
\varepsilon%
\end{pmatrix}%
\in\mathbb{T}\times [0,\varepsilon_1].
\end{equation*}
Then $U$ becomes a strongly continuous semiflow on the compact set $X:=%
\mathbb{T}\times [0,\varepsilon_1]$. Next consider the two positively
invariant sets (see \eqref{3.11})
\begin{equation*}
X^0:=\mathbb{T}^0\times [0,\varepsilon_1]\text{ and }\partial X^0=\partial%
\mathbb{T}^{0}\times [0,\varepsilon_1].
\end{equation*}

Now in order to prove the lemma, we will show that the pair $\left(\partial
X^0,X^0\right)$ is uniformly persistent with respect to the extended
semiflow $U$. To that aim, observe that $U$ possesses a compact global attractor, denoted by $%
A$. Then $U|_{\partial X^0}$ also admits a global attractor $%
A_\partial=\left([0,\gamma]\times \{0\}\right)\times [0,\varepsilon_1]$
while $\tilde A_\partial:=\bigcup_{Z\in A_\partial} \omega(X)$ can be decomposed as
the follows
\begin{equation*}
\tilde A_\partial=M_1\bigcup M_2\text{ with }M_1:=\left\{%
\begin{pmatrix}
0 \\
0%
\end{pmatrix}%
\right\}\times [0,\varepsilon_1]\text{ and }M_2:= \left\{%
\begin{pmatrix}
\gamma \\
0%
\end{pmatrix}%
\right\}\times [0,\varepsilon_1],
\end{equation*}
that corresponds to a covering of $\tilde A_\partial$ by disjoint compact
isolated invariant sets $M_1$ and $M_2$ for $U|_{\partial X^0}$. Furthermore
$M_1$ is chained to $M_2$ and this covering is acyclic (see \cite%
{Hale-Waltman}), since $\partial_U F_\varepsilon(0,0)>0$ and $\partial_U
F_\varepsilon(\gamma,0)<0$.

Next since $\{U(t)\}_{t\geq 0}$ is bounded dissipative and completely
continuous on $X$ for each $t\geq 0$, in view of Theorem 4.1 in \cite%
{Hale-Waltman} to prove that the pair $\left(\partial X^0,X^0\right)$ is
uniformly persistent, it is sufficient to check that
\begin{equation*}
W^s\left(M_i\right)\cap X^0=\emptyset,\;\forall i=1,2.
\end{equation*}
This latter property follows from the same repulsiveness arguments as the
ones developed in Proposition \ref{PROP2.5} using the inequalities in %
\eqref{3.12}.
\end{proof}

Using the above lemma one obtains the following decomposition result.

\begin{proposition}
\label{PROP3.7} For each $\varepsilon\in [0,\varepsilon_1]$, there exist a
global attractor $\mathcal{A}_{0,\varepsilon}\subset \mathbb{T}^0$ and a
global attractor $\mathcal{A}_{\partial,\varepsilon}$ in $\partial\mathbb{T}%
^0$ for $T_\varepsilon$ and the following decomposition for the global
attractor $\mathcal{A}_\varepsilon$ (provided by Lemma \ref{LE3.5}) holds
true
\begin{equation}  \label{3.14}
\mathcal{A}_\varepsilon=\mathcal{A}_{0,\varepsilon}\bigcup W^u\left(\mathcal{%
A}_{\partial,\varepsilon}\right),
\end{equation}
where $W^u\left(\mathcal{A}_{\partial,\varepsilon}\right)=\left\{X\in
\mathcal{A}_\varepsilon:\alpha(X)\subset \mathcal{A}_{\partial,\varepsilon}%
\right\}$. Furthermore the family $\left(\mathcal{A}_{0,\varepsilon}\right)_{%
\varepsilon\in [0,\varepsilon_0]}$ is upper semi-continuous.
\end{proposition}

\begin{proof}
The proof of the above result relies on the application of Theorem 3.2 in
\cite{Hale-Waltman} and Theorem 1.1 in \cite{Magal09}. To see this, let us
first observe that the result in Lemma \ref{LE3.6} can be reformulated as follows:
\begin{equation*}
\liminf_{t\to\infty}\mathrm{dist}\,\left(T_\varepsilon(t)X,\partial\mathbb{T}%
^0\right)\geq \Theta,
\end{equation*}
for all $X\in\mathbb{T}^0$ and all $\varepsilon\in [0,\varepsilon_1]$.
Hence, since for each $\varepsilon\in [0,\varepsilon_1]$, $T_\varepsilon$ is
completely continuous and bounded dissipative and satisfies \eqref{3.11}, the
existence $\mathcal{A}_{0,\varepsilon}$ $\mathcal{A}_{\partial,\varepsilon}$
together with the decomposition \eqref{3.13} follows from the results in
\cite{Hale-Waltman}. Next, using Lemma \ref{LE3.5} and \ref{LE3.6}, the
results of Magal in \cite{Magal09} applies and ensures the upper
semi-continuity for the family of interior attractors $\left\{\mathcal{A}%
_{0,\varepsilon}\right\}_{\varepsilon\in [0,\varepsilon_1]}$. This completes
the proof of the proposition.
\end{proof}

\begin{remark}
One may notice that, for all $\varepsilon\in [0,\varepsilon_1]$ one has
\begin{equation*}
\mathcal{A}_{\partial,\varepsilon}=[0,\gamma]\times \{0\}.
\end{equation*}
This point has -- implicitly -- already been used in the proof of Lemma \ref%
{LE3.6}.
\end{remark}

In the following, we discuss some properties of the interior attractor $%
\mathcal{A}_{0,\varepsilon}$ for $\varepsilon\in (0,\varepsilon_1]$. Our
first result consists in the perturbation of Theorem \ref{TH2.6} and it
reads as follows.

\begin{theorem}
Assume that
\begin{equation*}
\gamma \left( \beta -1\right) <\beta +1.
\end{equation*}
Then there exists $\varepsilon_2\in (0,\varepsilon_1]$ such that
\begin{equation*}
\mathcal{A}_{0,\varepsilon}=\left\{%
\begin{pmatrix}
\overline U_2 \\
\overline V_2%
\end{pmatrix}%
\right\},\;\forall \varepsilon\in [0,\varepsilon_2].
\end{equation*}
In other words, the interior attractor reduces to the interior equilibrium
for all $\varepsilon>0$ small enough.
\end{theorem}

\begin{proof}
The proof of this result relies on the application of Dulac's criterion.
Note that due to Lemma \ref{LE3.6}, one has
\begin{equation*}
\inf_{X\in \partial\mathbb{T}^0} \mathrm{dist}\,\left(X,\mathcal{A}%
_{0,\varepsilon}\right)\geq \Theta,\;\forall \varepsilon\in
[0,\varepsilon_0].
\end{equation*}
Let $K\subset \mathbb{T}$ be compact such that
\begin{equation*}
\inf_{X\in \partial\mathbb{T}^0} \mathrm{dist}\,\left(X,K\right)\geq \frac{%
\Theta}{2}\text{ and }A_{0,\varepsilon}\subset K\,\;\forall \varepsilon\in
[0,\varepsilon_0].
\end{equation*}
As for the proof of Theorem \ref{TH2.6}, we consider the function $\varphi(U,V)=\dfrac{1+U}{U}
V^{\xi-1}$. Then, since $(F_\varepsilon,G_\varepsilon)\to (F,G)$ as $%
\varepsilon\to 0$ for the topology of $C^1(\mathbb{T})$, one has
\begin{equation*}
\left[ \partial_U (\varphi F_\varepsilon)+\partial_V (\varphi G_\varepsilon) %
\right]\to \left[ \partial_U (\varphi F)+\partial_V (\varphi G) \right],
\end{equation*}
uniformly for $(U,V)\in K$ as $\varepsilon\to 0$. According to the
computations \eqref{Dulac} recalled in Theorem \ref{TH2.6} one has
$$
\max_{(U,V)\in K}\left[ \partial_U (\varphi F)+\partial_V (\varphi G) \right]<0.
$$
As a consequence, there exists $\varepsilon_2\in (0,\varepsilon_1]$ small enough and $\delta>0$ such that, for all $\varepsilon\in [0,\varepsilon_2]$ one has
\begin{equation*}
\left[ \partial_U (\varphi F_\varepsilon)+\partial_V (\varphi G_\varepsilon) %
\right]\leq -\delta,\;\forall (U,V)\in K.
\end{equation*}
Since $\mathcal{A}_{0,\varepsilon}\subset K$ for all $\varepsilon$ small
enough, the result follows using Dulac's criterion.
\end{proof}

\begin{lemma}
Assume that $\gamma \left( \beta -1\right) >\beta +1$. Then there exists $%
\varepsilon_3\in (0,\varepsilon_1]$ such that the interior equilibrium $%
\left(\overline U_2,\overline V_2\right)$ is an unstable spiral points for the
semiflow $T_\varepsilon$, for all $\varepsilon\in [0,\varepsilon_3]$. More
precisely, the linearized equation of system \eqref{3.11} around the interior
equilibrium has two complex conjugated eigenvalues with strictly positive
real parts, that is a two dimensional unstable manifold.
\end{lemma}

\begin{proof}
Consider the Jacobian matrix, denoted by $J_\varepsilon$, associated to %
\eqref{3.11} at $\left(\overline U_2,\overline V_2\right)$. Since $(F_\varepsilon,G_%
\varepsilon)$ is $C^1(\mathbb{T})-$close to $(F,G)$ as $\varepsilon\to 0$,
one has
\begin{equation*}
J_\varepsilon=J+o(1)\text{ as }\varepsilon\to 0.
\end{equation*}
Herein $J$ is the Jacobian matrix at $\left(\overline U_2,\overline V_2\right)$ of %
\eqref{3.11} with $\varepsilon=0$ (that corresponds to system \eqref{2.1}).
According to Lemma \ref{LE2.2}, the eigenvalues $\lambda_\pm$ of $J$ are
simple so that the eigenvalues of $J_\varepsilon$, $\lambda_{\pm,%
\varepsilon} $ are simple and continuous with respect to $\varepsilon$.
Hence $\lambda_{\pm,\varepsilon}=\lambda_{\pm}+o(1)$. This completes the
proof of the result since $\lambda_\pm$ are conjugated complex numbers with
positive real parts.
\end{proof}

Note that the system \eqref{3.11} has the same equilibria as system \eqref{2.1} and the system \eqref{3.11} has the boundary equilibria given by
\begin{equation*}
(\overline{U}_{0},\overline{V}_{0})=(0,0)\text{ and }(\overline{U}_{1},%
\overline{V}_{1})=(\gamma ,0)
\end{equation*}%
and the unique \textit{interior equilibrium} given by
\begin{equation*}
(\overline{U}_{2},\overline{V}_{2})=\left( \dfrac{1}{\beta -1},\dfrac{\alpha
\beta \left[ \gamma \left( \beta -1\right) -1\right] }{\left( \beta
-1\right) ^{2}}\right)
\end{equation*}%
whenever $\gamma \left( \beta -1\right) >1$.

As a consequence of the Poincar\'e-Bendixon theorem, one
obtains the following corollary.
\begin{corollary} \label{CO3.11}
Assume that $\gamma \left( \beta -1\right) >\beta +1$. Then there exists $%
\varepsilon_4\in (0,\varepsilon_3]$ such that for all $\varepsilon\in
[0,\varepsilon_4]$, the interior attractor $\mathcal{A}_{0,\varepsilon}$
contains a (non-trivial) periodic orbit surrounding the interior equilibrium
$\left(\overline U_2,\overline V_2\right)$.
\end{corollary}

\subsection{Uniqueness of the periodic orbit and interior attractor}

In this section we discuss the uniqueness of the periodic orbit
for system \eqref{3.11} and its relationship with the global interior attractor when the parameters satisfy the condition
\begin{equation}\label{cond}
\gamma \left( \beta -1\right) >\beta +1\;
\Leftrightarrow\;
\overline{U}_2 <\dfrac{\gamma-1}{2}.
\end{equation}

The aim of this section is to prove the following uniqueness result.

\begin{theorem}[Unique stable periodic orbit]
\label{TH3.12} Under condition \eqref{cond}, for all $\varepsilon>0$ small enough,
 there exists a unique stable periodic orbit surrounding the interior
equilibrium and the system has no other periodic orbit.
\end{theorem}

According to Corollary \ref{CO3.11}, for each $\varepsilon>0$ small enough, let $(U_\varepsilon(t),V_\varepsilon(t))$ denotes any non constant periodic orbit of \eqref{3.11} and $T_\varepsilon>0$ its period. The associated closed curve is denoted by $\Gamma_\varepsilon$, that is
$$
\Gamma_\varepsilon=\left\{\left(U_\varepsilon(t),V_\varepsilon(t)\right),\;t\in [0,T_\varepsilon]\right\}.
$$
Recall that $\Gamma_\varepsilon$ encloses the interior equilibrium $(\overline U_2,\overline V_2)$.
Note also that $\Gamma_\varepsilon\subset \mathcal A_{0,\varepsilon}$. Hence Proposition \ref{PROP3.7} ensures that there exists $\theta>0$ such that for all $\varepsilon>0$ small enough
$$
U_\varepsilon(t)\geq \theta,\;V_\varepsilon(t)\geq \theta,\;\forall t\in\R.
$$
Throughout this section we also denote by $\Gamma_0$ the unique non constant periodic orbit of \eqref{2.1}, that corresponds to \eqref{3.11} with $\varepsilon=0$ (see Theorem \ref{TH2.8}). The corresponding periodic solution of \eqref{2.1} is denoted by $(U_0(t),V_0(t))$ while $T_0>0$ is its period.

The idea of this proof is to show that $\Gamma_\varepsilon$ becomes close to $\Gamma_0$ as $\varepsilon\to 0$.
Then, as in \cite{Cheng81} for the unperturbed system, we will prove that for all $\varepsilon>0$ small enough,
\begin{equation*}
\int_0^{T_\varepsilon} \left[\partial_U F_\varepsilon(U_\varepsilon(t),V_\varepsilon(t))+\partial_V G_\varepsilon(U_\varepsilon(t),V_\varepsilon(t))\right]dt<0.
\end{equation*}
According to Hale \cite{Hale}, the latter condition means that $\Gamma_\varepsilon$ is locally asymptotically stable and then it follows that $\Gamma_\varepsilon$ is unique when $\varepsilon>0$ is small enough.

To prove Theorem \ref{TH3.12}, let us firstly prove the following lemma.

\begin{lemma}\label{LE3.13} Let condition \eqref{cond} be satisfied. Then, for each $\delta\in \left(\overline U_2,\frac{\gamma-1}{2}\right)$, there exists $\varepsilon(\delta)>0$ small enough such for all $\varepsilon\in (0,\varepsilon(\delta)]$ the curve $\Gamma_\varepsilon$ intersects the line $U=\delta$. In other words,  one has $\max\left\{U_\varepsilon(t):\;t\in [0,T_\varepsilon]\right\}\geq \delta$ for all $\varepsilon\in (0,\varepsilon(\delta)]$.
\end{lemma}

\begin{proof}
Consider the function
\begin{equation*}
\mathcal{F}  (U,V)=\int_{\overline{U}_{2}}^{U}\frac{\left( \dfrac{\beta \xi }{%
1+\xi }-1+\Psi_\varepsilon(\xi,V)\right) }{\dfrac{\beta \xi }{1+\xi }}d\xi +\frac{1}{\beta }\int_{%
\overline{V}_{2}}^{V}\frac{\eta -\overline{V}_{2}}{\eta }d\eta,
\end{equation*}
wherein we have set $\Psi_\varepsilon(U,V)=-V^{-1}\Phi^2_\varepsilon(U,V)$. Note that since $\Phi_\varepsilon(U,V)$ is $C^1-$small uniformly on the compact set $(U,V)\in\mathbb T$ with $V\geq \theta>0$ and $U\geq \theta>0$ then $\Psi_\varepsilon$ is also $C^1-$small on the same compact set.

Next let us compute the derivative of function $\mathcal{F}  (U_\varepsilon,V_\varepsilon)$ with respect to $t$ along the periodic orbit $\Gamma_\varepsilon$, that yields
\begin{equation*}
\begin{split}
\dfrac{d\mathcal{F}  (U_\varepsilon(t),V_\varepsilon(t))}{dt}&=\displaystyle \frac{\left( \dfrac{\beta U_\varepsilon}{1+U_\varepsilon}-1+\Psi_\varepsilon(U_\varepsilon,V_\varepsilon)\right) }{\dfrac{\beta U_\varepsilon}{1+U_\varepsilon}}\left[
-\Phi^1_{\varepsilon }(U_\varepsilon,V_\varepsilon)+F(U_\varepsilon,V_\varepsilon)\right] +\frac{1}{\beta }\frac{V_\varepsilon-%
\overline{V}_{2}}{V_\varepsilon} V_\varepsilon' \\
&+V_\varepsilon'\int_{\overline{U}_{2}}^{U_\varepsilon}\frac{(1+\xi)\partial_V \Psi_\varepsilon(\xi,V_\varepsilon) }{\beta \xi }d\xi.
\end{split}
\end{equation*}
This rewrites as
\begin{equation*}
\begin{split}
\dfrac{d\mathcal{F}  (U_\varepsilon,V_\varepsilon)}{dt}&=\displaystyle \frac{V_\varepsilon'}{V_\varepsilon}\frac{1+U_\varepsilon}{\beta U_\varepsilon}\left[
-\Phi^1_{\varepsilon }(U_\varepsilon,V_\varepsilon)+F(U_\varepsilon,V_\varepsilon)+\frac{U_\varepsilon}{1+U_\varepsilon}(V_\varepsilon-\overline V_2)\right] \\
&+V'\int_{\overline{U}_{2}}^{U_\varepsilon}\frac{1+\xi}{\beta\xi}\partial_V \Psi_\varepsilon(\xi,V_\varepsilon)d\xi,
\end{split}
\end{equation*}
and denoting by $\tilde \Psi_\varepsilon(U,V)=-\frac{1+U}{\beta U}\Phi^1_{\varepsilon }(U,V)$, this yields
\begin{equation*}
\dfrac{d\mathcal{F}  (U_\varepsilon,V_\varepsilon)}{dt}=\displaystyle \frac{V_\varepsilon'}{\beta V_\varepsilon}\left[\tilde \Psi_\varepsilon((U_\varepsilon,V_\varepsilon)+f(U_\varepsilon)-\overline V_2\right] +V_\varepsilon'\int_{\overline{U}_{2}}^{U_\varepsilon}\frac{\left(\partial_V \Psi_\varepsilon(\xi,V_\varepsilon)\right) }{\dfrac{\beta \xi }{1+\xi }}d\xi.
\end{equation*}
Integrating the above equality on $[0,T_\varepsilon]$ leads
\begin{equation*}
0=\int_0^{T_\varepsilon}\dfrac{d\mathcal{F}  (U_\varepsilon(s),V_\varepsilon(s))}{dt}ds=\oint_{\Gamma_\varepsilon}\left\{\displaystyle \frac{1}{\beta V}\left(\tilde\Psi_\varepsilon(U,V)+f(U)-\overline V_2\right) +\int_{\overline{U}_{2}}^{U}\frac{1+\xi}{\beta \xi}\partial_V \Psi_\varepsilon(\xi,V)d\xi\right\}dV.
\end{equation*}
Now denoting by $\Omega_\varepsilon$ the interior of the periodic curve $\Gamma_\varepsilon$ and using the Green-Riemann formula, we infer the following identity
\begin{equation}\label{id}
0=\int_{\Omega_\varepsilon}\left\{\displaystyle \frac{1}{\beta V}\left(\partial_U \tilde\Psi_\varepsilon(U,V)+f'(U)\right) +\frac{1+U}{\beta U}\partial_V \Psi_\varepsilon(U,V)\right\}dUdV.
\end{equation}
Now fix $\delta\in \left(\overline U_2,\frac{\gamma-1}{2}\right)$ and recall that $\inf_{{U\leq \delta}} f'(U)=f'(\delta)>0$.
Set $K=\{(U,V)\in\mathbb T:\;U\in [\theta,\delta],\;V\geq \theta\}$ and observe that $\partial_V \Psi_\varepsilon$ and $\tilde \Psi_\varepsilon$ tend to $0$ as $\varepsilon\to 0$, uniformly for $(U,V)\in K$. Hence since $K$ is bounded by some constant $R>0$, we obtain uniformly for $(U,V)\in K$ and for all $0<\varepsilon\ll 1$
\begin{equation*}
\frac{1}{V}\left(\partial_U \tilde\Psi_\varepsilon(U,V)+f'(U)\right) +\frac{1+U}{\beta U}\partial_V \Psi_\varepsilon(U,V)\geq f'(\delta)+o(1).
\end{equation*}
Hence there exists $\varepsilon(\delta)>0$ such that for all $\varepsilon\in (0,\varepsilon(\delta)]$ one has
$$
\sup_{(U,V)\in K}\displaystyle \frac{1}{V}\left(\partial_U \tilde\Psi_\varepsilon(U,V)+f'(U)\right) +\frac{1+U}{\beta U}\partial_V \Psi_\varepsilon(U,V)>0.
$$
As a consequence, since $\Gamma_\varepsilon\cup \Omega_\varepsilon\subset K$ and $\Gamma_\varepsilon$ encloses the equilibrium, if the curve $\Gamma_\varepsilon$ does not intersect the line $U=\delta$ for all $\varepsilon>0$ small enough then the integral on the right hand side of \eqref{id} would be positive which is a contradiction and we complete the proof of the lemma.
\end{proof}

We continue the proof of Theorem \ref{TH3.12} by showing the following lemma.

\begin{lemma} \label{LE3.14} Let condition \eqref{cond} be satisfied. Let $\Gamma_0$ denote the unique non-constant periodic orbit of system \eqref{2.1}. Then the following convergence holds
$$
\lim_{\varepsilon \to 0} {\rm d}(\Gamma_\varepsilon, \Gamma_0)=0
$$
where ${\rm d}(\Gamma_\varepsilon, \Gamma_0)$ denotes the Hausdorff's semi-distance given by
$$
{\rm d}(\Gamma_\varepsilon, \Gamma_0):=\sup_{x \in \Gamma_\varepsilon} \delta(x,\Gamma_0)
\text{ with }
\delta(x,\Gamma_0)= \inf_{y \in \Gamma_0} \Vert x-y \Vert.
$$
In other words,  for each neighborhood $V$ of $\Gamma_0$ there exists $\varepsilon_V >0$ such that
$$
\Gamma_\varepsilon \subset V, \forall \varepsilon \in (0,\varepsilon_V].
$$
Furthermore the period $T_\varepsilon>0$ of $\Gamma_\varepsilon$ converges to $T_0$, the period of $\Gamma_0$, as $\varepsilon \to 0$.
\end{lemma}

\begin{proof}
Fix $U^\star=\frac{1}{2}\left[\overline{U}_2+\frac{\gamma-1}{2}\right)\subset \left(\overline{U}_2,\frac{\gamma-1}{2}\right)$.

\noindent \textit{Step 1:}  From Lemma \ref{LE3.13} for all $\varepsilon>0$ small enough, there exists $t_\varepsilon \in \mathbb{R}$ such that
$$
U_{\varepsilon}(t_\varepsilon)>U^\star.
$$

\noindent \textit{Step 2:}  Using Arzela-Ascoli's theorem we can find a sequence  $\varepsilon_n \to 0$ and $t \to \left( U(t),V(t)\right)$ a complete orbit of the unperturbed system \eqref{2.1} such that
\begin{equation}
\left( U_{\varepsilon_n}(t+t_\varepsilon),V_{\varepsilon_n}(t+t_\varepsilon) \right) \to
\left( U(t),V(t)\right)
\end{equation}
for the topology of the local uniform convergence for $ t \in \mathbb{R}$. The definition of $t_\varepsilon$ above ensures that
\begin{equation}\label{non-deg}
U(0)>U^\star .
\end{equation}
Moreover, since  $(U_\varepsilon (t),V_\varepsilon (t))\in \mathbb T$ and $U_\varepsilon (t)\geq \theta$ and $V_\varepsilon (t)\geq \theta$ for all $\varepsilon$ small enough and $\forall t \in \mathbb{R}$, one obtains that
$$
(U,V)(t)\in \mathbb T,\;\;\forall t \in \mathbb{R} \text{ and }U(t)\geq \theta,\;V(t)\geq \theta,\;\forall t\in \mathbb{R}.
$$
Hence the limit orbit $(U,V)$ lies in the interior attractor $\mathcal{A}_{\mathrm{Int} \left( \mathbb{R}^2_+\right)}$ of \eqref{2.1} while \eqref{non-deg} implies that the complete orbit is not reduced to the interior equilibrium, therefore
$$
\lim_{t \to \infty} \delta \left( \left( U(t), V(t) \right),\Gamma_0 \right)=0.
$$

\noindent \textit{Step 3:} Let us fix $M^0=\left(U^0,V^0 \right) \in \Gamma_0$ such that
$$
F(M^0)>0 \text{ and }G(M^0)>0.
$$
In order to simplify the rest of the proof, we fix the norm $\|\cdot\|_1$ in $\R^2$ given by
$$
\Vert\left(U,V \right) \Vert_1=\vert U\vert+\vert V\vert,\;\forall (U,V)\in\R^2.
$$
Let $\eta>0$ be small enough and let $\varepsilon_0=\varepsilon_0(\eta)>0$ be small enough (depending on $\eta$) such that
$$
F_\varepsilon(M)>F(M^0)/2 \text{ and } G_\varepsilon(M)>G(M^0)/2,
$$
whenever $\Vert M-M^0 \Vert_1 \leq \eta$ and $\varepsilon \in (0,\varepsilon_0)$.

By using the sign of $F$ and $G$ around $M^0$, we can find $M^1=(U^1,V^1)$ a point on $\Gamma_0$ such that
$$
M^0<M^1 (\text{that is}, U^0<U^1\text{ and } V^0<V^1 )\text{ and } \Vert M^1-M^0 \Vert_{1} <  \eta.
$$
Let $\delta \in (0,\eta)$ be such that
$$
\left\lbrace M \in \mathbb{R}^2: \Vert M-M^1 \Vert_1 \leq \delta \right\rbrace \subset \left\lbrace M \in \mathbb{R}^2: M \geq M^0 \text{ and  }  \Vert M-M^0 \Vert_1 \leq \eta\right\rbrace.
$$

\noindent \textit{Step 4:} By using the continuous dependency of the semiflow generated \eqref{3.11} with respect to the initial condition and with respect to the parameter $\varepsilon$ we deduce that we can find $\widehat{\delta}\in (0,\delta)$ and $\varepsilon_1 \in (0, \varepsilon_0)$  such that every solution of  \eqref{3.11} starting in the ball
$$
B(M^1,\widehat{\delta}):= \left\lbrace M \in \mathbb{R}^2: \Vert M-M^1 \Vert_1 \leq \widehat{\delta}\right\rbrace
$$
will belong to the larger ball
$$
B(M^1,\delta):=\left\lbrace M \in \mathbb{R}^2: \Vert M-M^1 \Vert_1 \leq \delta\right\rbrace
$$
at time $t=T_0$.

\noindent \textit{ Step 5:} By using the Step 2, for all $n$ large enough, we find $M_{\varepsilon_n} \in \Gamma_{\varepsilon_n}$ belonging in the ball $B(M^1,\widehat{\delta})$ and the solution of the approximated system \eqref{3.11} starting from $M_{\varepsilon_n} $ belongs to the ball $B(M^1,\delta)$ at $t=T_0$.

Assume by contradiction that this solution leaves the triangle
$$
T=\left\lbrace  M \in \mathbb{R}^2: M\geq M^0 \text{ and }\Vert M-M^0 \Vert_1 \leq \eta \right\rbrace
$$
without intersecting the point  $M_{\varepsilon_n} $. By using Jordan's theorem, we obtain a contradiction since the closed curve $\Gamma_{\varepsilon_n}$ cannot return back through the triangle $T$ from the "exit segment"
$$
S=\left\lbrace  M \in \mathbb{R}^2:  M\geq M^0 \text{ and }\Vert M-M^0 \Vert_1=\eta  \right\rbrace.
$$
This completes the proof of the lemma.
 \end{proof}

We now complete the proof of Theorem \ref{TH3.12} by proving, announced above that for all $\varepsilon>0$ small enough
 \begin{equation*}
\int_0^{T_\varepsilon} \left[\partial_U F_\varepsilon(U_\varepsilon(t),V_\varepsilon(t))+\partial_V G_\varepsilon(U_\varepsilon(t),V_\varepsilon(t))\right]dt<0.
\end{equation*}
However this estimate follows from some properties of the unique periodic orbit $(U_0,V_0)$ of \eqref{2.1} together with the convergence result stated in Lemma \ref{LE3.14}.
Indeed, note that Cheng \cite{Cheng81} proved that, the unique unperturbed periodic orbit $\Gamma_0$ satisfies
  \begin{equation*}
\int_0^{T_0} \left[\partial_U F(U_0(t),V_0(t))+\partial_V G(U_0(t),V_0(t))\right]dt<0,
\end{equation*}
 while Lemma \ref{LE3.14} ensures that
  \begin{equation*}
\lim_{\varepsilon\to 0}\int_0^{T_\varepsilon} \left[\partial_U F_\varepsilon(U_\varepsilon(t),V_\varepsilon(t))+\partial_V G_\varepsilon(U_\varepsilon(t),V_\varepsilon(t))\right]dt=\int_0^{T_0} \left[\partial_U F(U_0(t),V_0(t))+\partial_V G(U_0(t),V_0(t))\right]dt<0.
\end{equation*}
 This completes the proof of the estimate and thus the one of Theorem \ref{TH3.12}.

As a consequence of the above result, we now can state the following properties of the interior attractor $\mathcal{A}_{0,\varepsilon}$ for all $0<\varepsilon\ll 1$.
\begin{theorem}
\label{TH3.13}
Assume that $\gamma \left( \beta -1\right) >\beta +1$. Then for all $\varepsilon>0$ small enough, the interior global attractor $\mathcal{A}_{0,\varepsilon}$ consists of the unique interior equilibrium $\left(\overline U_2,\overline V_2\right)$ and the interior of the unique periodic orbit surrounding the interior equilibrium, and an infinite number of heteroclinic orbits joining the unique interior equilibrium and the unique periodic orbit.
\end{theorem}

\subsection{Existence and uniqueness of a traveling wave joining $(\gamma ,0)$ and the interior global attractor}

In this section, we use the previous results to provide a description of the heteroclinic orbits for \eqref{3.11} as well as their uniqueness.

\begin{lemma}
\label{LE3.14bis}
Assume that $\gamma \left( \beta -1\right) >1$. Then, for all $\varepsilon>0$ small enough the equilibria $\left(0,0\right)$ and $\left(\gamma,0\right)$ are saddle points for the
semiflow $ T_\varepsilon$. More
precisely, the linearized equation of system \eqref{3.11} around the equilibrium $\left(0,0\right)$ (or $\left(\gamma,0\right)$) has one eigenvalue with positive real part and one with negative real part.
\end{lemma}

\begin{proof}
Let us denote by $J_\varepsilon$ the Jacobian matrix associated to %
\eqref{3.11} at $\left(0,0\right)$. Since $(F_\varepsilon,G_%
\varepsilon)$ is $C^1(\mathbb{T})-$close to $(F,G)$ as $\varepsilon\to 0$,
one has
\begin{equation*}
J_\varepsilon=J+o(1)\text{ as }\varepsilon\to 0,
\end{equation*}
where $J$ denotes the Jacobian matrix at $\left(0,0\right)$ of %
\eqref{3.11} with $\varepsilon=0$ (that corresponds to system \eqref{2.1}).
It is easy to check that the eigenvalues of $J$ are the following: $\lambda_{+,J}=\alpha \gamma>0$ and $\lambda_{-,J}=-1<0$. The eigenvalues $\lambda_{\pm,J_\varepsilon}$ of $J_\varepsilon$ are continuous with respect to $\varepsilon$.
Hence $\lambda_{\pm,J_\varepsilon}=\lambda_{\pm,J}+o(1)$. This completes the
proof of the result.
\end{proof}

\begin{proposition}
\label{PROP3.15} Assume that $1<\gamma \left( \beta -1\right)$. Then system \eqref{3.11} admits a unique heteroclinic orbit going from $(0,0)$ to $(\gamma ,0)$, for all $0<\varepsilon\ll 1$ small enough.
\end{proposition}

\begin{proof}
Since $\partial_u\mathbb{T}$ is positively invariant with respect to the system \eqref{3.11} and
\begin{equation*}
F(U,V) =\alpha U(\gamma -U)-\dfrac{UV}{1+U}=\frac{U}{1+U}[f(U)-V],
\end{equation*}
by using the following fact
\begin{equation*}
\lim_{\varepsilon \to 0 }\Vert \Phi_\varepsilon \Vert_{\infty}=0,
\end{equation*}
we can deduce that there exists a unique heteroclinic orbit of system \eqref{3.11} going from $(0,0)$ to $(\gamma ,0)$.
\end{proof}

\vspace{1ex}

We now discuss the existence and uniqueness of heteroclinc orbits for \eqref{3.11} joining the boundary to the interior attractor. As in the previous section, we make use of the connectedness of the global attractor to derive the existence of such connections. We then discuss further properties.

\noindent \textbf{Connectedness arguments:} The largest global attractor $\mathcal{A}_{\varepsilon}$  is connected. Since any continuous map maps a connected set into a connected set, it follows that the projection of $\mathcal{A}_{\varepsilon}$  on the horizontal and vertical axis is a compact interval.

The global attractor $\mathcal{A}_{\varepsilon}$ contains the interior global attractor $\mathcal{A}_{0,\varepsilon}$ which is compacts connected and locally stable and also contains the boundary attractor $\mathcal{A}_{\partial, \varepsilon}$. The connectedness of $\mathcal{A}_{\varepsilon}$ and compactness  of $\mathcal{A}_{0,\varepsilon}$ and $\mathcal{A}_{\partial,\varepsilon}$ imply
$$
\mathcal{A}_{\varepsilon}  - \mathcal{A}_{0,\varepsilon} \bigcup  \mathcal{A}_{\partial, \varepsilon} \neq \emptyset.
$$
Moreover by using Proposition 3.7, we deduce that for each point $(U,V) \in  \mathcal{A}_{\varepsilon}  - \mathcal{A}_{0,\varepsilon} \bigcup  \mathcal{A}_{\partial, \varepsilon}$
the $\alpha$ and $\omega$ limit sets satisfy the following
$$
\alpha(U,V) \in \mathcal{A}_{\partial, \varepsilon} \text{ and } \omega(U,V) \in \mathcal{A}_{0,\varepsilon}.
$$
Finally since the boundary attractor has a Morse decomposition $M_1=\left\{ (0,0) \right \}$ and $M_2=\left\{ (\gamma,0) \right \}$ we have either
$$
\alpha(U,V) =M_1 \text{ or } \alpha(U,V) =M_2, \forall (U,V) \in  \mathcal{A}_{\varepsilon}  - \mathcal{A}_{0,\varepsilon} \bigcup  \mathcal{A}_{\partial, \varepsilon}.
$$

\begin{proposition}
\label{PROP3.16} Assume that $1<\gamma \left( \beta -1\right)$.
There for all $\varepsilon$ small enough, system \eqref{3.11} does not admit any heteroclinic orbit going from $(0,0)$ to the interior global attractor.
\end{proposition}

\begin{proof}
Assume by contradiction that there exists one. Note that
$$
 V^{\prime }  =  -\Phi^2_\varepsilon(U,V)+\left( -1+ \dfrac{\beta U}{1+U} \right)V.%
$$
We deduce that
$$
V(t)=\exp \left( \int_{t_0}^t  -1+ \dfrac{\beta U(s)}{1+U(s)}  ds \right) \left( V(t_0)+\int_{t_0}^t -\Phi^2_\varepsilon(U(t),V(t)) \exp \left( \int_{t_0}^s  1- \dfrac{\beta U(t)}{1+U(t)}  dt \right) ds\right).
$$
Since there exists $T<0$ such that $U(t)$ remains sufficiently small for all negative times $t<T$ and $V(t_0)>0$, we deduce that
$$
\lim_{t \to -\infty}V(t)=+\infty
$$
which contradicts the fact that the solution belongs to the global attractor and is therefore bounded.
\end{proof}

We complete this section by proving the uniqueness of the traveling wave solution connecting $(\gamma,0)$ to the interior global attractor. The arguments of this proof extend those used in \cite{Ducrot-Langlais-Magal}.

\begin{proposition}
\label{PROP3.17} Assume that $1<\gamma \left( \beta -1\right)$. Then for all $\varepsilon>0$ small enough, system \eqref{3.11} admits a unique heteroclinic orbit going
from $(\gamma ,0)$ to the interior global attractor.
\end{proposition}

\begin{proof}
We only need to prove the uniqueness. From Lemma 3.16, it follows that the center-unstable manifold at $(\gamma ,0)$ is a one dimensional locally invariant manifold. By using the same arguments as in section 2 for the uniqueness of the heteroclinic orbit starting from $(\gamma,0)$ for system (2.1), we can prove the uniqueness of the heteroclinic orbit going from $(\gamma ,0)$ to the interior global attractor for system \eqref{3.11}.
\end{proof}

\section{Numerical simulations}
In this section we intend to observe the previous results numerically. We run some numerical simulations for the system
\begin{equation}
\left\lbrace
\begin{array}{ll}
U_{t} &=dU_{xx}+\alpha U(\gamma -U)-\dfrac{UV}{1+U}, \text{ for } x \in [0,1000] \\
V_{t} &=V_{xx}-V+\beta \dfrac{UV}{1+U},  \text{ for } x \in [0,1000]
\end{array}
\right.
\end{equation}
with Neumann boundary conditions
$$
U_x(t,x)=V_x(t,x)=0, \text{ for } x=0 \text{ and } x=1000,
$$
and the initial values
$$
U(0,x)=\gamma \text{ and } V(0,x)=0.1*\exp\left(-\delta x \right).
$$
Throughout the simulations the parameters will be unchanged and fixed as follows
$$
d=1,\,\alpha=1/4, \, \gamma=4, \, \beta=2.
$$

In Figure \ref{fig2}, we observe the traveling wave joining $(\gamma,0)$ and periodic wave train when we start from a $V(0,x)$ with $\delta=0.1$.

\begin{figure}[H]
		\centering
		\includegraphics[width=2.8in,height=2in]{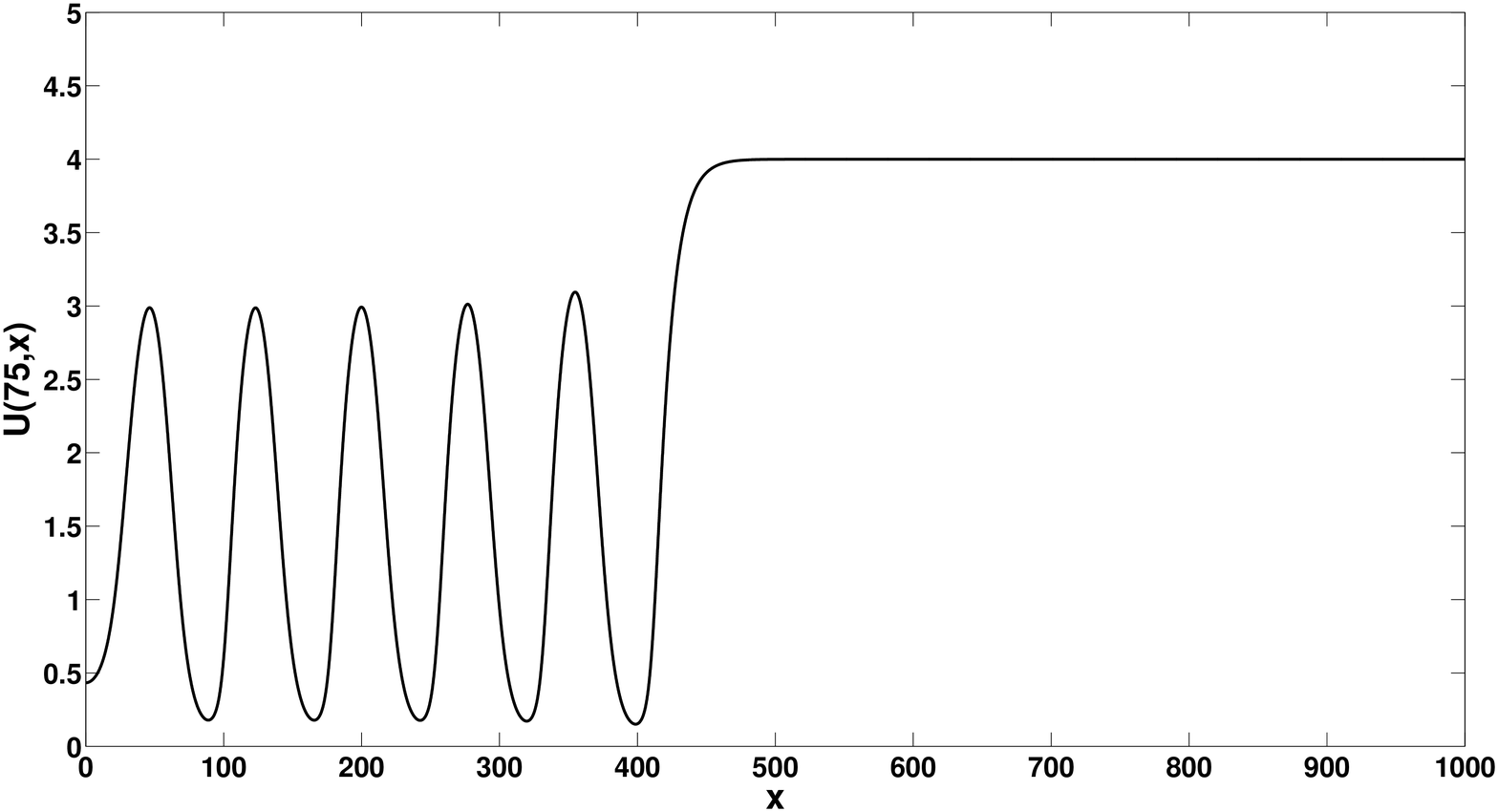}
		\includegraphics[width=2.8in,height=2in]{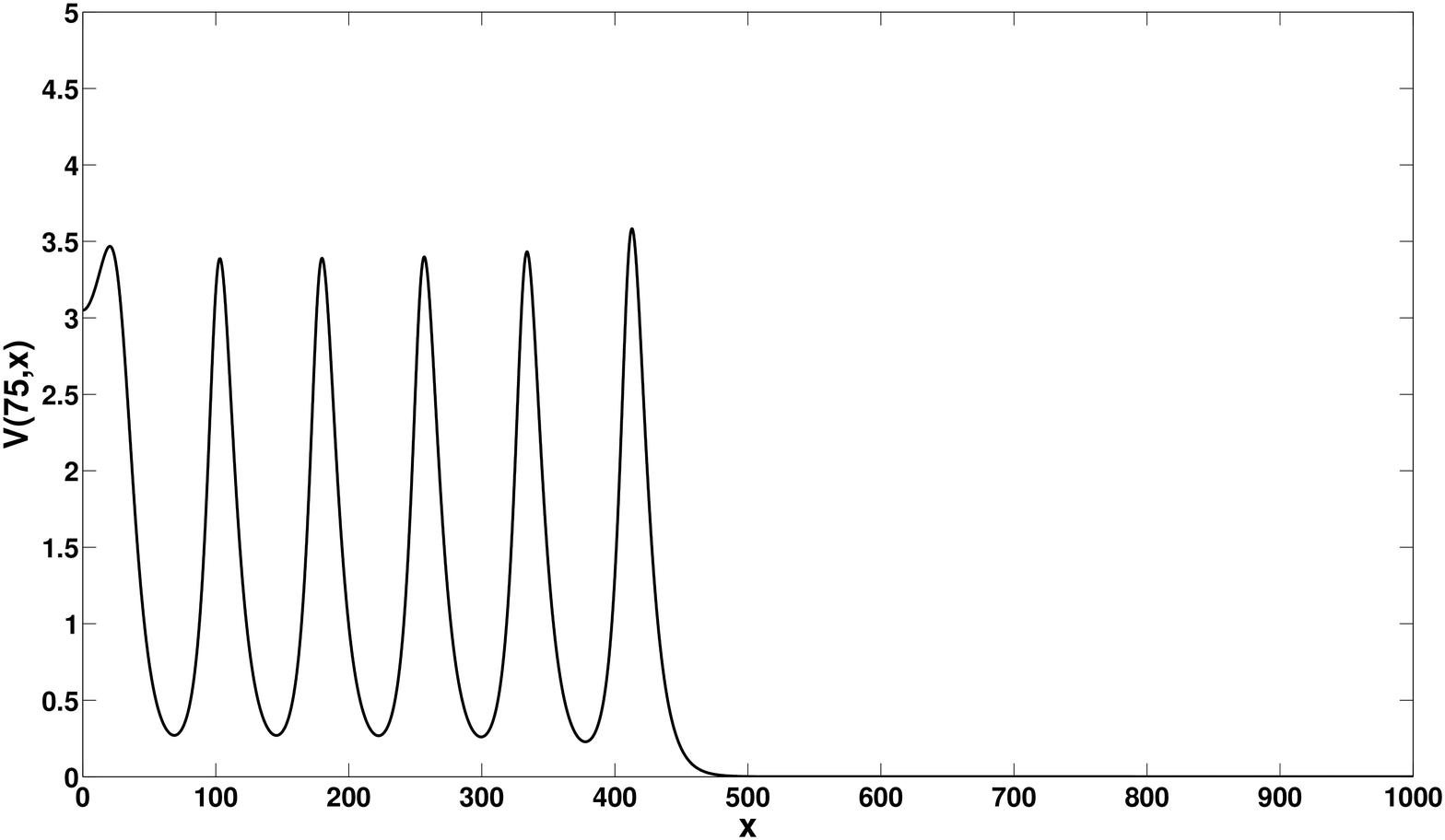}
		\includegraphics[width=2.8in,height=2in]{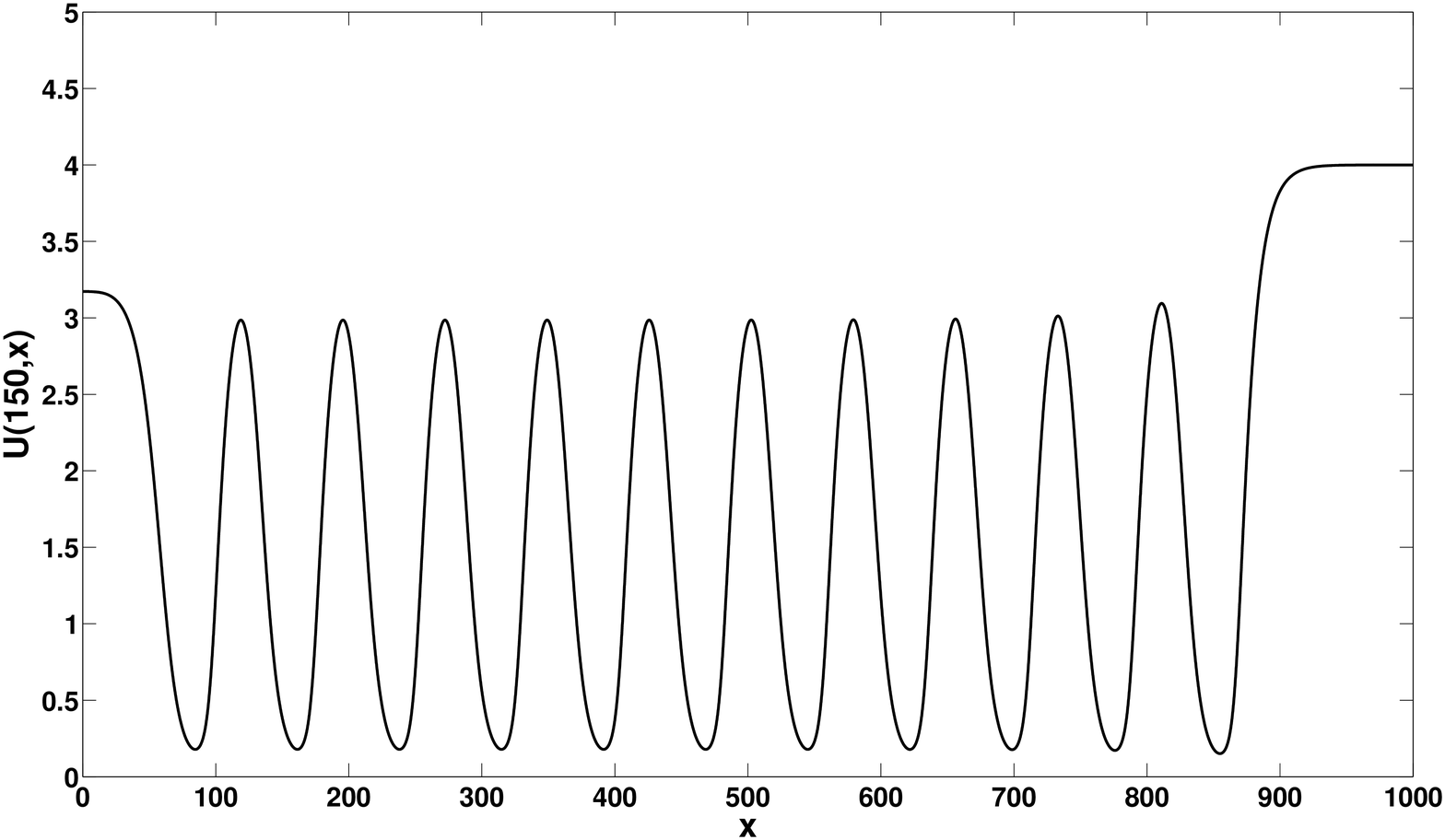}
		\includegraphics[width=2.8in,height=2in]{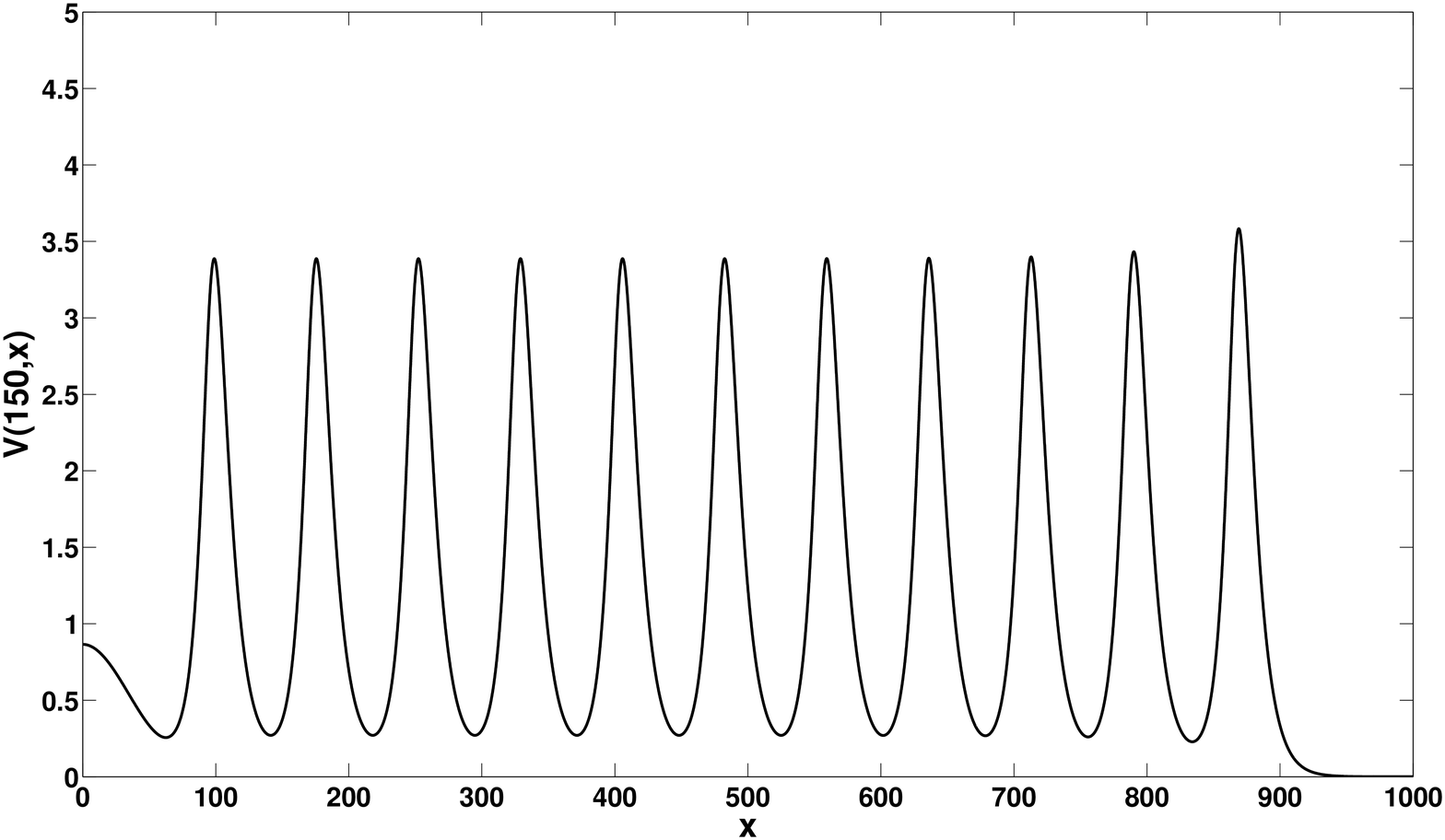}
		\caption{\textit{In  this figure we plot $U(t,x)$ (left handside) and $V(t,x)$ (right handside) whenever the parameter $\delta=0.1$ and $t=75$ (above) and $t=150$ (below). The initial distribution $U(0,x)=\gamma$ and $V(0,x)=0.1*\exp\left(-0.1 x \right)$.  We observe a traveling wave joining $(\gamma,0)$ and a periodic waves train with both predator and prey oscillating periodically.}}\label{fig2}
\end{figure}	

In Figure \ref{fig3}, we observe some more complex behaviours whenever we start from a $V(0,x)$ with $\delta=1$. The complexity in such a problem was already observed by Sherratt, Smith and Rademacher \cite{Sherratt-Smith-Rademacher} in the multi-dimensional case.

\begin{figure}[H]
		\centering
		\includegraphics[width=2.8in,height=2in]{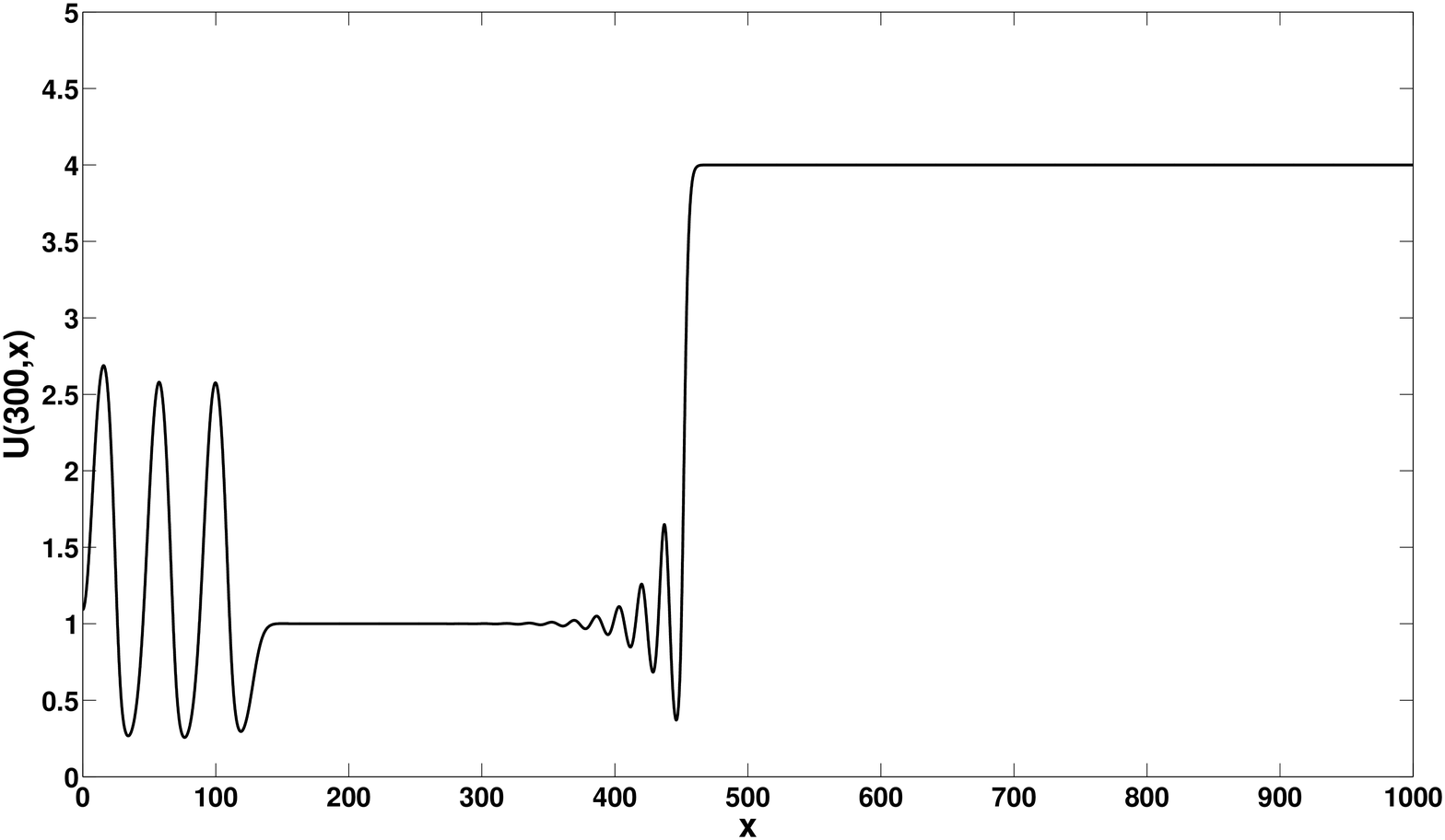}
		\includegraphics[width=2.8in,height=2in]{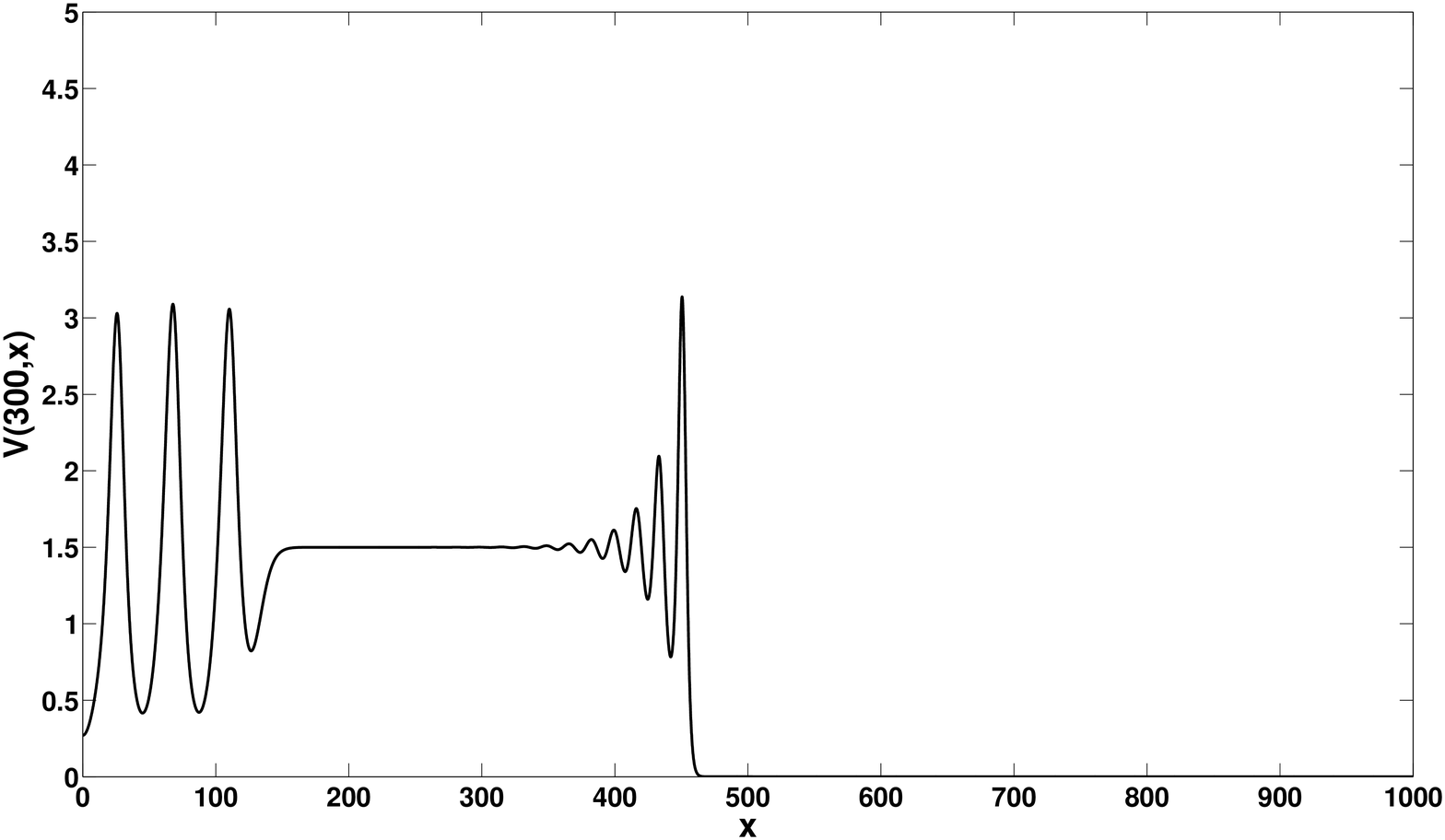}
		\includegraphics[width=2.8in,height=2in]{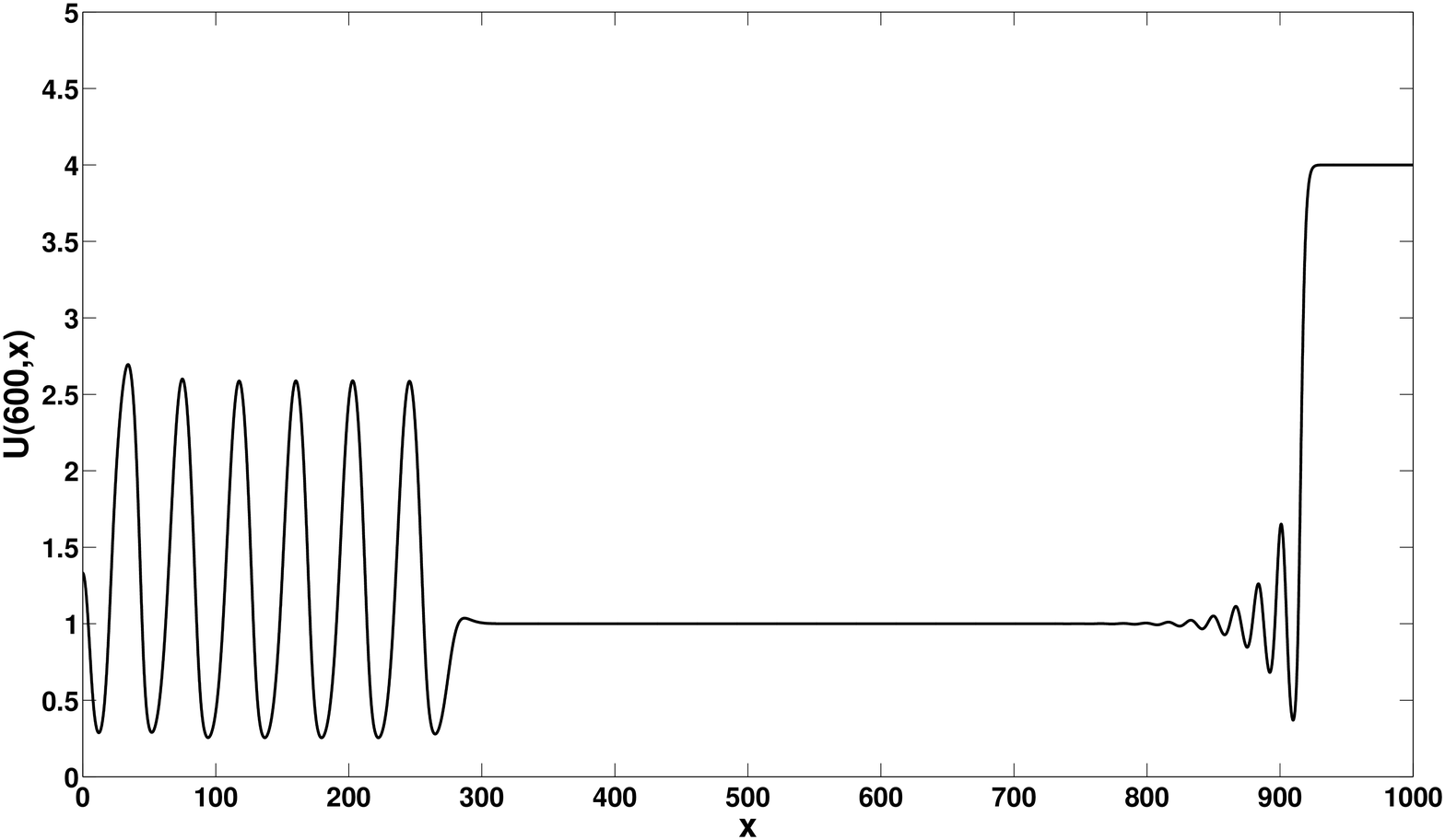}
		\includegraphics[width=2.8in,height=2in]{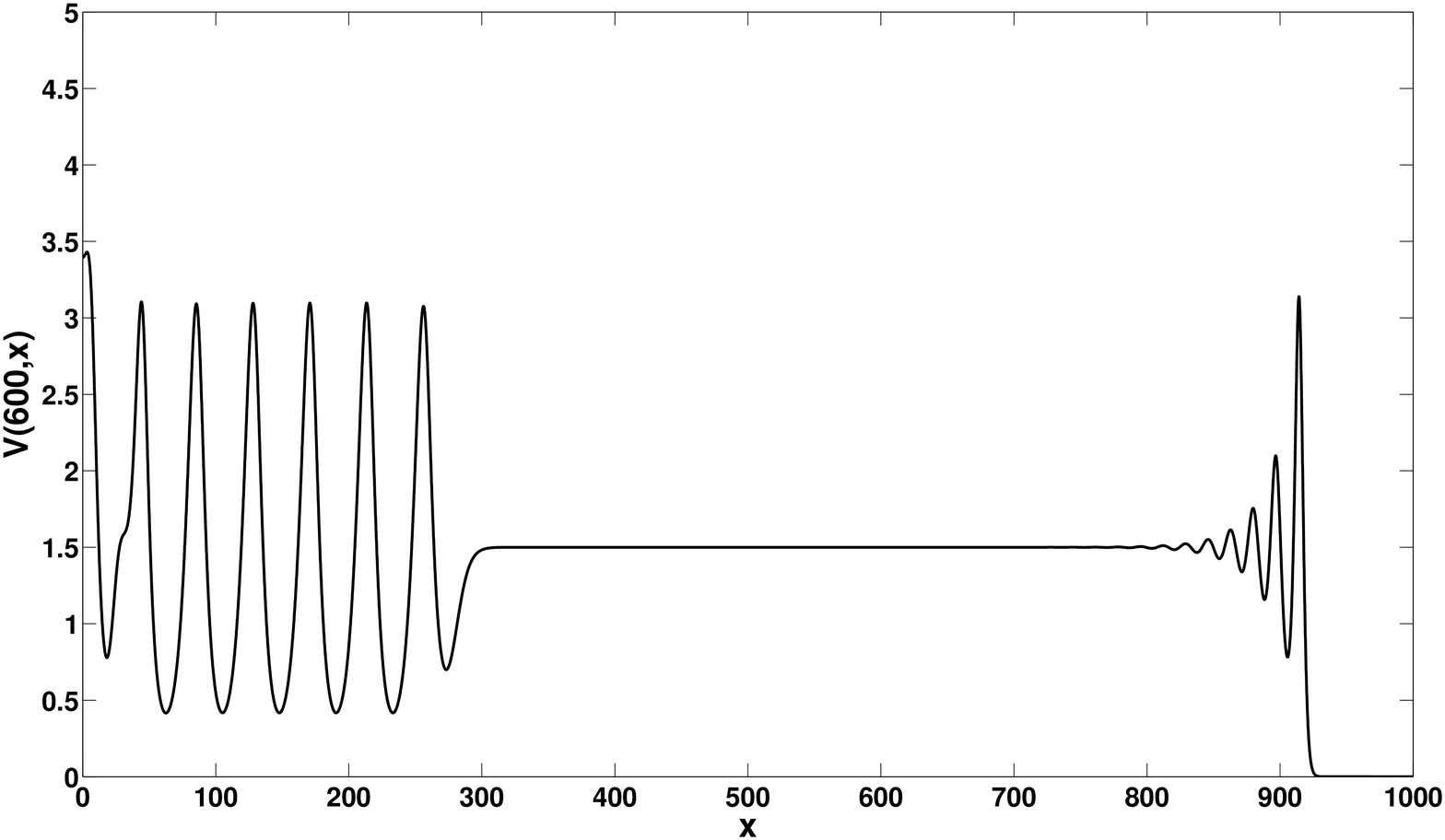}
		\caption{\textit{In this figure we plot $U(t,x)$ (left handside) and $V(t,x)$ (right handside) whenever the parameter $\delta=1$ and $t=300$ (above) and $t=600$ (below). The initial distribution $U(0,x)=\gamma$ and $V(0,x)=0.1*\exp\left(- x \right)$.  We observe a traveling wave joining $(\gamma,0)$ and the positive equilibrium superposed with a periodic traveling pulse.}}\label{fig3}
\end{figure}

%

\end{document}